\documentclass[reqno]{amsart}
\usepackage{amsmath,amsbsy,amsfonts}
\usepackage{color}
\def\Xint#1{\mathchoice
{\XXint\displaystyle\textstyle{#1}}%
{\XXint\textstyle\scriptstyle{#1}}%
{\XXint\scriptstyle\scriptscriptstyle{#1}}%
{\XXint\scriptscriptstyle\scriptscriptstyle{#1}}%
\!\int}
\def\XXint#1#2#3{{\setbox0=\hbox{$#1{#2#3}{\int}$ }
\vcenter{\hbox{$#2#3$ }}\kern-.6\wd0}}

\def\dashint{\Xint-}

\usepackage[utf8]{inputenc}
\usepackage{graphicx}
\usepackage{amsfonts,amsfonts}
\usepackage{empheq,amsmath,latexsym,amssymb}

\usepackage{color}
\newcommand\blfootnote[1]{%
  \begingroup
  \renewcommand\thefootnote{}\footnote{#1}%
  \addtocounter{footnote}{-1}%
  \endgroup
}

\newtheorem{theorem}{Theorem}[section]
\newtheorem{proposition}{Proposition}[section]
\newtheorem{corollary}[proposition]{Corollary}
\newtheorem{definition}[proposition]{Definition}
\newtheorem{lemma}[proposition]{Lemma}
\newtheorem{remark}[proposition]{Remark}

\newcommand{\diver}{\operatorname{div}}
\newcommand{\dd}{\mathrm{d}}
\newcommand{\osc}{\operatorname{osc}}

\numberwithin{equation}{section}


\begin{document}
\title[Fractional $(p,q)$-Laplacian]{Local behaviour and existence of solutions of the fractional $(p,q)$-Laplacian}
	
\author{Emerson Abreu}
\address{E. Abreu and A. H. S. Medeiros - Departmento de Matem\'atica, Universidade Federal de  Minas Gerais, 31270-901 - Belo Horizonte - MG, Brazil}
\email{\begin{align*} eabreu@ufmg.br&\\   
aldomedeiros@ufmg.br&\end{align*}}
\author{A. H. Souza Medeiros}

\date{}

 
\begin{abstract}
In this paper, we consider the regularity of weak solutions (in an appropriate space) to the elliptic partial differential equation
\begin{equation*}
(-\Delta_{p})^{s} u + (-\Delta_{q})^{s} u = f(x) \quad \text{in} \quad \mathbb{R}^{N},
\end{equation*}
where $0<s<1$ and $ 2 \leq q \leq p < N/s$. We prove that these solutions are
 locally in $C^{0,\alpha}(\mathbb{R}^N)$, which seems to be optimal. Furthermore, we prove the existence of solutions to the problem
\begin{equation*}
(-\Delta_{p})^{s} u + (-\Delta_{q})^{s} u = \vert u \vert^{p^{*}_{s}-2}u + \lambda g(x) \vert u \vert^{r-2}u \,\,\, \text{in} \,\,\,\, \mathbb{R}^{N},
\end{equation*}
where $1 < q\leq p < N/s$, $\lambda$ is a parameter and  $g$ satisfies some  conditions of integrability. We also show that, if $g$ is bounded, then the solutions are continuous and bounded.
\end{abstract}

	\maketitle
\section{Introduction}

\blfootnote{This study was financed by the Conselho Nacional de Desenvolvimento Cient\'{\i}fico e Tecnol\'ogico - Brasil (CNPq) and Funda\c c\~ao de Apoio \`a Pesquisa do Estado de Minas Gerais - (FAPEMIG)}
We investigate the existence and regularity of weak solutions of the $(p,q)$-Laplacian problem 
\begin{equation}\label{P_1}\left\{
\begin{array}{rcll}
(- \Delta_{p})^{s}u + (- \Delta_{q})^{s}u =f \quad \mbox{in} \quad \mathbb{R}^{N}\vspace*{.1cm} \\ 
u(x) \geq 0 , \quad x \in \mathbb{R}^{N},
\end{array} \right.
\end{equation}
where $0 < s < 1$, $N > sp$, $p^{*}_{s} = \frac{Np}{N-sp}$, $2 < q \leq p < \infty$  and  $f \in L^{\frac{p^{*}_{s}}{p^{*}_s-1}}(\mathbb{R}^{N}) \cap L^{\theta}(\mathbb{R}^{N})$, with $\theta > \frac{N}{sp}$. The hypothesis $f\in L^{\frac{p^{*}_{s}}{p^{*}_s-1}}(\mathbb{R}^{N})$ ensure that we can apply variational methods, and the condition $f \in L^{\theta}(\mathbb{R}^{N})$ is necessary to apply the Moser iteration process to obtain a bound in $L^\infty$-norm for a solution. 
	
For any $1\leq m < \infty$, the fractional $m$-Laplacian operator, under suitable smoothness condition on $\phi$, can be written as
\begin{equation}\label{nonlocal}
(-\Delta_{m})^{s} \phi (x) = 2 \displaystyle\lim_{\varepsilon \rightarrow 0} \displaystyle\int_{\mathbb{R}^{N} \setminus B_{\varepsilon}(x)} \frac{\vert \phi(x) - \phi(y)\vert^{m-2}(\phi(x) - \phi(y))}{\vert x - y\vert^{N+ s m}} \dd y, \ \ \ \forall x \in \mathbb{R}^{N},
\end{equation}
where $B_{\varepsilon}(x) := \{y \in \mathbb{R}^{N}; \vert y-x\vert < \varepsilon \}$, see \cite{Franzina,Palatucci,Pucci}.
	
We point out that there are several notions of the fractional Laplacian operator in the current literature, all of which agree when the problems are set on the whole $\mathbb{R}^N$. However, some of them disagree in a bounded domain.
	
Recently, a great deal of attention has been focused on studying problems involving fractional operators, from a pure mathematical point of view and for applications as well, since this kind of problem naturally arise in many different contexts, such as the thin obstacle problem, finance, phase transitions, stratified materials, optimization, anomalous diffusion, semipermeable membranes, minimal surfaces, among others. For more details, see \cite{Caffarelli,Caisheng,Palatucci,Yang}. 
	
The regularity up to the boundary of fractional problems like \eqref{P_1} in the case $p=q=2$ is now rather well understood, even when more general kernels and nonlinearities are considered. Using a viscosity solution approach, the linear model case gives regularity for fully non-linear equations which are ``uniformly elliptic” in a suitable sense. Regarding the viscosity approach to fully non-linear, elliptic non-local equation. See \cite{Caffar} for interior regularity theory with smooth kernels.
	
In the case $p,q \neq 2$, problem \eqref{P_1} is both non-local and non-linear. Its leading operator $(-\Delta_{p})^{s}$ is furthermore degenerate when $p>2$. Determining sufficiently good regularity estimates up to the boundary is not only relevant by itself, but it also has useful applications in obtaining multiplicity results for more general non-linear and non-local equations, such as those investigated by Ianizzotto, Liu, Perera and Squassina \cite{Liu} in the framework of topological methods and Morse theory. 
	
The $C^{0,\alpha}$-regularity of weak solutions of the degenerate elliptic problem
\begin{equation}
\left\{
\begin{array}{rcll}
(-\Delta_{p})^{s} u = f \quad \text{in} \quad \Omega,\vspace*{.1cm} \\ 
u = 0 , \quad \text{in} \quad \Omega^{c},  \end{array} \right.
\end{equation}
when  $1 < p < \infty$ was proved by Ianizzotto, Mosconi and Squassina \cite{Ian}, for a bounded domain $\Omega \subset \mathbb{R}^{N}$ with $C^{1,1}$-boundary and $f \in L^{\infty}(\Omega)$.
	
When $s=1$, \eqref{P_1} becomes a $(p,q)$-Laplacian problem of the form
\begin{equation}
(-\Delta_{p}) u + (-\Delta_{q}) u= f(x), \quad x \in \mathbb{R}^N, \label{PL}
\end{equation}
which has its origin in the general reaction-diffusion problem
\begin{equation}
u_t = \diver(D(u)\nabla u) + f(x,u), \quad x \in \mathbb{R}^N, \ \ t>0, \label{Difusao}
\end{equation}
where $D(u) = \vert \nabla u \vert^{p-2} + \vert \nabla u \vert^{q-2}$.
	
For a general term $D(u)$, problem \eqref{Difusao} has a wide range of applications in physics and related sciences such as biophysics, plasma physics, and chemical reaction design. In such applications, the function $u$ describes a concentration, and the first term on the right-hand side of \eqref{Difusao} corresponds to a diffusion with a diffusion coefficient $D(u)$; the term $f(x,u)$ stands for the reaction, related to sources and energy-loss processes. Typically, in chemical and biological applications, the reaction term $f(x,u)$ is a polynomial in $u$ with variable coefficients (see \cite{Figueiredo,Liang,Yang}).
	
The regularity of solutions of \eqref{PL} has been studied by He and Li \cite{Gongbao}. The authors showed that the weak solutions are locally $C^{1,\alpha}$.

The first difficulty found in problem \eqref{P_1} was how to define a weak solution. We address this question in the present paper. For this purpose, we usually consider the reflexive Banach space
\begin{equation*}
\mathcal{W}:= D^{s,p}(\mathbb{R}^N) \cap D^{s,q}(\mathbb{R}^N),
\end{equation*}
where $D^{s,m}(\mathbb{R}^N ) = \{u \in L^{m_{s}^*}(\mathbb{R}^N) ; \Vert u \Vert_{ s,m} < \infty \}$  and $\Vert u \Vert_{s,m}$ denotes the Gagliardo-norm 
\begin{equation*}
\Vert u \Vert_{s,m} = \left( \displaystyle\int_{ \mathbb{R}^N}\displaystyle\int_{ \mathbb{R}^N} \frac{ \vert u(x)-u(y) \vert^{m}}{\vert x-y \vert^{N+ s m}} \dd x\dd y \right)^{\frac{1}{m}}
\end{equation*}
for all $u \in D^{s,m}(\mathbb{R}^N)$. See \cite{Brasco} for  details.
	
The non-homogeneity of the operator $(-\Delta_{p})^{s} + (-\Delta_{q})^{s}$ introduces technical difficulties in obtaining and regularity of weak solutions for problems involving this operator. It is worth to mention that $\|\cdot\|_{s,m}$ is a norm in $D^{s,m}(\mathbb{R}^{N})$, but not in $W^{s,m}(\mathbb{R}^{N})$. \\
	
Our first and main result is concerned with local regularity of weak solutions of problem \eqref{P_1}:
\begin{theorem}\label{BOUNDED}
Let $\theta > \frac{N}{sp}$, $f \in L^{\frac{p^{*}_{s}}{p^{*}_s-1}}(\mathbb{R}^{N}) \cap L^{\theta}(\mathbb{R}^{N})$ and $u \in \mathcal{W}$ a solution of \eqref{P_1}. Then $u \in L^{\infty}(\mathbb{R}^{N})$. 
		
Moreover, if $f \in L_{loc}^{\infty}(\mathbb{R}^{N})$, then $u$ is locally H\"{o}lder continuous with exponent $\alpha$, that is, $u\in C_{loc}^{\alpha}(\mathbb{R}^{N})$ 
with $\alpha \in \left(0,\frac{s(p-q)}{p-1}\right)$.
\end{theorem}
	
The additional condition $f \in L_{loc}^{\infty}(\mathbb{R}^{N})$ in the above theorem is used to control the oscillations of $u$ in a ball.
	
To prove the global boundedness of the solution $u$, in Section \ref{bound} we use the Moser iteration process. The continuity of the solution $u$ is obtained by adapting the arguments used by Ianizzotto, Mosconi and Squassina \cite{Ian} and Serrin \cite{Serrin}, which was done in Section \ref{SEC}. The main idea is to control the oscillation of the function $u$ in any ball. In order to do so, we prove a Harnack type inequality for weak solutions of \eqref{P_1} instead of viscosity solutions, since we consider that the variational setting is more natural to the problem. However, barrier type arguments are frequently used in our approach. Since this kind of argument is not valid if $1<p<2$, our proof only applies if $2\leq q,p<\infty$. \\
	
We will also study existence and regularity of weak solutions for the problem involving the fractional critical $p^{*}_{s}$-exponent. 
\begin{equation}\label{EQ}\left\{
\begin{array}{rcll}
(- \Delta_{p})^{s}u + (-\Delta_{q})^{s}u =\vert u \vert^{p^{*}_{s}-2}u + \lambda g\vert u \vert^{r-2}u\, \quad \mbox{in} \quad \mathbb{R}^{N} \\ 
u(x) \geq 0 , \quad  x \in \mathbb{R}^{N}  \end{array} \right.
\end{equation}
where $0 < s < 1$, $N > sp$, $p^{*}_{s} = \frac{Np}{N-sp}$, $1 < q \leq p < \infty$ and $g$ satisfies the following integrability conditions:
	
\begin{enumerate}
\item[$(g_1)$] $g$ is integrable and $g \in L^{t_s}( \mathbb{R}^N )$, with $t_s = \frac{p^{*}_{s}}{p^{*}_{ s} - r}$;
\item[$(g_2)$] There exist an open set $\Omega_{g} \in \mathbb{R}^N$ and $\alpha_0 >0$ such that $g(x) \geq \alpha_0 > 0$, for all $x \in \Omega_{g}$.
\end{enumerate}

Considering problem \eqref{EQ}, we prove:
\begin{theorem}\label{EXISTENCE}
Assume that $g:\mathbb{R}^N \rightarrow \mathbb{R}$ satisfies the conditions $(g_1)$ and $(g_2)$, for $0<s<1$. 
\begin{enumerate}		
\item [$(i)$] If $1< q \leq p <r < p_{s}^{*}$, then there exists $\lambda^* >0$ such that, for any $\lambda > \lambda^*$, problem \eqref{EQ} has at least one nontrivial and nonnegative weak solution in $\mathcal{W}$.
		
\item [$(ii)$] If $1<q<\displaystyle\frac{N(p-1)}{N-s} < p \leq \max\left\{p,p_{s}^* - \frac{q}{p-1} \right\} < r < p_{s}^* $ and $ N>p^2s $, then the problem \eqref{EQ} has a non-trivial weak solution in $\mathcal{W}$ for any $\lambda > 0$.
\end{enumerate}
\end{theorem}
	
\begin{theorem}\label{APLIC}
Let $2 < q \leq p < r < p_{s}^*$, $\lambda > 0$ and $0 < s < 1$, be such that $N > sp$. Assume that $g \in L^{\infty}(\mathbb{R}^N)$ satisfy $(g_1)$ and  $u \in D^{s,p}(\mathbb{R}^N) \cap  D^{s,p}(\mathbb{R}^N)$ is a solution of \eqref{EQ}. Then $u \in L^{\infty}(\mathbb{R}^N)  \cap C^{\alpha}_{loc}(\mathbb{R}^{N})$.
\end{theorem}
	
When $s = 1$ , \eqref{EQ} is reduced for the $(p,q)$-Laplacian equation
\begin{equation}
(-\Delta_{p}) u + (-\Delta_{q}) u= \vert u \vert^{p^* -2}u + \lambda g(x) \vert u\vert^{r-2}u, \ \ \ x \in \mathbb{R}^N. \label{PL1}
\end{equation}
	
The existence of a nontrivial solution of problem \eqref{PL1} has been studied by Chaves, Ercole and Miyagaki in \cite{Marcio}. They showed the existence of a non-trivial solution if $\lambda$ is large enough. Using the theory of regularity developed by  He and  Li in \cite{Gongbao}, they showed that the weak solutions are locally $C^{1,\alpha}$, if $g \in L^{t_1}(\mathbb{R}^{N}) \cap L^{\infty}(\mathbb{R}^{N})$.
	
In Section \ref{SEC1}, motivated by \cite{Marcio} we will assume that $1< q \leq p <r < p_{s}^{*}$ and  investigate the existence of a nontrivial solution for the problem \eqref{EQ}. We show that for $\lambda$ large enough exists a solution of problem \eqref{EQ}. 

Furthermore, assuming that $1<q<\frac{N(p-1)}{N-s} < p \leq \max\left\{p,p_{s}^* - \frac{q}{p-1} \right\} < r < p_{s}^* $, and $ N>p^2s $, we show that \eqref{EQ} has a solution for any $\lambda >0$ by using estimates for the extremal function (see \cite{IND,Brasco,Mos}).
	
In order to obtain a nontrivial solution of \eqref{EQ}, we apply a version of the Mountain Pass Theorem (see \cite{Springer}). 	
	
We also adapt standard arguments to prove the boundedness of Palais-Smale sequences. In order to overcome the lack of compactness of Sobolev's embedding, we prove a pointwise convergence result, which together with the Brezis-Lieb lemma give us the weak convergence. Following arguments similar to \cite{Marcio,GB,Yin}, in Section \ref{SEC1} we obtain a strict upper bound for $c_{\lambda}$, the level of the Palais-Smale sequence, valid for all $\lambda$ large enough. Applying this fact and arguments adapted from \cite{Marcio,Alves} to conclude that the nonnegative critical points for $I_{\lambda}$(the associated euler lagrange functional), obtained from the mountain pass theorem are not the trivial ones.
	
Taking advantage of the compact embedding $W_{0}^{ s,p}(\Omega)\hookrightarrow L^{t}(\Omega)$ for $1 \leq t < p_s^*$, we can study the fractional $p$-Laplacian problems in  bounded domains. However, the situation is quite different for the $(p,q)$-Laplacian, since the embedding $W^{s,p}(\Omega) \hookrightarrow W^{s,q}(\Omega)$, for $\Omega \subset \mathbb{R}^{N}$ with $p\neq q$ does not always exist [cf. \cite{ContraEx}], which is a additional difficulty.

When the domain is the whole $\mathbb{R}^N$, the Sobolev embedding is not compact. To work around this problem, a concentration-compactness principle or minimization restricted methods (see \cite{Wang,Binlin}) have been used to find weak solutions in $W^{s,p}(\mathbb{R}^N )$. 

Finally, in Section \ref{prelim}, we recollect some basic fact about the fractional framework that will be very important in the paper.
\section{Preliminaries}\label{prelim}
\subsection{Functions spaces} \mbox{}
For all measurable function $u: \mathbb{R}^{N} \rightarrow \mathbb{R}$, let 
\begin{align*}
[ u ]_{s,m,\Omega} = \left(\displaystyle\int_{\Omega}\displaystyle\int_{\Omega} \frac{|u(x) - u(y) |^{m}}{|x-y|^{N+sm}} \dd x\dd y\right)^{1/m} 
\end{align*}
be the Gagliardo semi-norm. We will consider the following spaces (see \cite{Adams,Palatucci,BP} for details):
\begin{align*}
W^{s,m}(\Omega) &= \{ u \in L^{m}(\Omega)\ ; \,\, [ u ]_{s,m,\Omega}  < \infty \}, 
\end{align*}
equipped with the norm
\begin{align*}
\| u \|_{s,m}=\|u\|_{W^{s,m}(\Omega)} = \|u\|_{L^{m}(\Omega)} +[ u ]_{s,m,\Omega} ,
\end{align*}
and 
\begin{align*}
W^{s,m}_{0}(\Omega) &= \{u \in W^{s,m}(\mathbb{R}^N); \,\, u =0 \ \text{in} \ \mathbb{R}^N\backslash\Omega\},\\
W^{-s,m'}(\Omega) &= \left(W^{s,m}(\Omega) \right)^*,\ m'=\frac{m}{m-1}\ (\mbox{dual space}).
\end{align*}
For any $1<m<\frac{N}{s}$ define the  reflexive Banach space
\begin{align*}
D^{s,m}(\mathbb{R}^{N}) := \{u\in L^{m^{*}_{s}}(\mathbb{R}^{N}); [ u ]_{s,m} 
 < \infty \}
\end{align*} 
where $m^{*}_{s} = \frac{Nm}{N-sm}$. The so-called best Sobolev constant for the embedding $D^{s,m}(\mathbb{R}^{N})\hookrightarrow L^{m^{*}_{s}}(\mathbb{R}^{N})$ is given by
\begin{align}\label{S}
S_{s,m} = \displaystyle\inf_{u \in D^{s,m}(\mathbb{R}^N) \backslash \{0 \}} \frac{ [ u ]_{s,m} 
^m}{ \Vert u \Vert_{m^*_{s}}^m}.
\end{align}
See \cite{Brasco} for details.
	
We will frequently make use of the following space (See \cite{Ian}):
\begin{definition}
Let $\Omega \subset \mathbb{R}^N$ be bounded. We set \footnote{$\Omega\Subset U$ means that $\Omega$ is a compact subset of $U$.}
\begin{equation*}
\tilde{W}^{s,m}(\Omega) := \left\lbrace u \in L_{loc}^{m}(\mathbb{R}^N) : \exists \ U \Supset  \Omega,\, \,  \Vert u \Vert_{W^{s,p}(U)} + \displaystyle\int_{\mathbb{R}^N} \frac{\vert u(x) \vert^{p-1}}{(1 + \vert x \vert)^{N + sp}} \dd x < \infty \right\rbrace.
\end{equation*}
		
If $\Omega$ is unbounded, we set 
\begin{equation*}
\tilde{W}_{loc}^{s,m}(\Omega) := \{u \in L_{loc}^{m}(\mathbb{R}^N) : u \in \tilde{W}^{s,m}(\Omega') \ \mbox{for any bounded} \ \Omega' \subseteq \Omega \}.
\end{equation*}
\end{definition}
For all $\alpha \in (0,1]$ and all measurable $u : \overline{\Omega} \rightarrow \mathbb{R}$ we set
\begin{align*}
&\vert u\vert_{C^{\alpha}(\overline{\Omega})} = \displaystyle\sup_{x,y \in \overline{\Omega},x\neq y} \frac{\vert u(x) - u(y) \vert}{\vert x-y\vert^{\alpha}} \\
&C^{0,\alpha}(\overline{\Omega}) = \{u \in C(\overline{\Omega}): |u|_{C^{\alpha}(\overline{\Omega})} < \infty \}.
\end{align*}
Throughout the paper we assume that $0<\alpha<1$ and
\[C^{\alpha}(\overline{\Omega})=C^{0,\alpha}(\overline{\Omega}),\]
being a Banach space under the norm 
\[\Vert u \Vert_{C^{\alpha}(\overline{\Omega})} = \Vert u\Vert_{L^\infty{\Omega})} + \vert u \vert_{C^{\alpha}(\overline{\Omega})}.\]
	
We recall that the $nonlocal \ tail$ centered at $x \in \mathbb{R}^N$ with radius $R > 0$, is defined as
\begin{align}
Tail_{m}(u;x;R) = \left(R^{sm} \displaystyle\int_{B_{R}^{c}(x)} \frac{\vert u(y)\vert^{m-1}}{\vert x - y \vert^{N+sm}} \dd y \right)^{1/(m-1)}.
\end{align}
	
We will also set $Tail_m(u;0;R) = Tail_m(u;R)$.
\begin{remark}\label{OBSFUN}
Note that, if $u \in L_{\infty}(\mathbb{R}^{N})$ then
\begin{equation*}
Tail_m(u;R)^{m-1} = R^{sm} \displaystyle\int_{B_{R}^{c}(0)} \frac{\vert u(y)\vert^{m-1}}{\vert y \vert^{N+sm}} \dd y \leq \frac{N \omega_{N} \left\Vert u \right\Vert_{\infty}^{m-1}}{sm}.
\end{equation*}
Thus, if $u \in L_{\infty}(\mathbb{R}^{N})$ we have $Tail_m(u;R)^{m-1} \leq C$, where $C>0$ is independent of $R$.
\end{remark}
	
\subsection{Some elementary inequalities} \mbox{}
	
For all $t \in \mathbb{R}$, we set
\[
J^{m} = \vert t \vert^{m-2} t.
\]	
We recall a few well-known inequalities
\begin{align}
(a+b)^m &\leq 2^{m-1}(a^m + b^m), \ a,b\geq 0, \ m \geq 1; \label{inq1}\\
(a+b)^m &\leq a^m + b^m \  a,b\geq 0,\ m\in (0,1]; \label{inq2}\\
\big\vert J_{m+1}(a) - J_{m+1}(b) \big\vert &\leq m\big(J_{m}(a) + J_{m}(b)\big)\vert a-b\vert, \ a,b \in \mathbb{R},\ q \geq 1. \label{inq3}
\end{align}	
Using the Taylor's formula and Young's inequality, we can prove that, for all $\theta>0$ exists $C_{\theta}>0$ such that
\begin{align}
(a + b)^{q} - a^q \leq \theta a^q + C_{\theta} b^q,  \  a,b \geq 0, \  q>0, \ \mbox{and} \ C_{\theta} \rightarrow \infty \ \mbox{as} \ \theta \rightarrow 0^{+}. \label{inq4}
\end{align}
Consider, for $b>0$, the function $f(t) = J_{m}(t) - J_{m}(t-b)$. Its global minimum is attained on $f(b/2) = 2^{2-m}b^{m-1}$, and hence we obtain the inequality
\begin{align}
J_{m}(a) - J_m(a-b) \geq 2^{2-m}b^{m-1}, \ \forall a\in \mathbb{R}, \ b \geq 0, \ \mbox{and}\ q\geq 1. \label{inq5}
\end{align}
	
Finally, in order to apply Moser iteration process we will use the following lemma:
\begin{lemma}\label{Le}
Let $1 < m < \infty$ and $g:\mathbb{R}\longrightarrow \mathbb{R}$ be a increasing function. Defining
\[
G(t) = \displaystyle\int_{0}^{t} (g'(\tau))^{\frac{1}{m}} \dd\tau, \ t\in \mathbb{R},
\]
we have that
\[
J_m(a-b)(g(a) - g(b)) \geq \vert G(a) - G(b) \vert^{m}.
		\]
\end{lemma}
\textbf{Proof:} See \cite[Lemma A.2]{BP}
\subsection{Some basic properties of the fractional $(p,q)$-Laplacian} \mbox{}
The following result describes a fundamental non-local feature of the fractional $(p,q)$-Laplacian operator $(-\Delta_p)^s + (-\Delta_q)^s$.
	
Given $1\leq q\leq p < \infty$ and $\Omega \subset \mathbb{R}^{N}$ denote by
\begin{equation*}
\mathcal{W}(\Omega) = \tilde{W}^{s,p}(\Omega) \cap \tilde{W}^{s,q}(\Omega).
\end{equation*}
\begin{definition}\label{DEF1}
Let $\Omega \subset \mathbb{R}^{N}$ be bounded and $u \in \mathcal{W}(\Omega)$. We say that $u$ is \textit{weak solution} of $(-\Delta_p)^{s} u + (-\Delta_{q})^{s}u = f$ in $\Omega$ if, for all $\varphi \in C_{0}^{\infty}(\Omega)$,
\begin{equation*}
\displaystyle\sum_{m=p,q}\displaystyle\int_{\mathbb{R}^{N}}\displaystyle\int_{\mathbb{R}^{N}} \frac{J_{m}(u(x) - u(y))(\varphi(x) - \varphi(y))}{|x-y|^{N+sm}} \dd x\dd y = \displaystyle\int_{\Omega}f \varphi \dd x
\end{equation*} 
		
The inequality $(-\Delta_{p})^{s}u + (-\Delta_q)^{s}u \leq f$ weakly in $\Omega$ will mean that
\begin{equation*}
\displaystyle\sum_{m=p,q}\displaystyle\int_{\mathbb{R}^{N}}\displaystyle\int_{\mathbb{R}^{N}} \frac{J_{m}(u(x) - u(y))(\varphi(x) - \varphi(y))}{|x-y|^{N+sm}} \dd x\dd y \leq \displaystyle\int_{\Omega}f \varphi \dd x
\end{equation*}
for all $\varphi \in C^{\infty}_{0}(\Omega)$, $\varphi \geq 0$. Similarly for $(-\Delta_{p})^{s}u + (-\Delta_q)^{s}u \geq f$. 
\end{definition}
\begin{remark}
By Lemma 2.3 in \cite{Ian} the functional
\begin{align*}
W^{s,m}_{0}(\Omega) \ni \varphi \mapsto (u,\varphi):= \displaystyle\int_{ \mathbb{R}^N} \displaystyle\int_{ \mathbb{R}^N} \frac{J_{m}(u(x)-u(y))(\varphi(x)-\varphi(y))}{|x-y|^{N+sm}}\dd x\dd y
\end{align*}
is finite and belongs to $W^{-s.m'}(\Omega)$, which implies that the Definition \ref{DEF1} makes sense.
\end{remark}
	
\begin{lemma}\label{lem1}
Suppose that $u \in \mathcal{W}(\Omega)$ satisfies $(-\Delta_p)^su + (-\Delta_q)^su = f$ weakly in $\Omega$ for some $f \in L^{1}_{loc}(\Omega)$. Let $v \in L^{1}_{loc}(\mathbb{R}^N)$ be such that 
\begin{align*}
\textup{dist}(\textup{supp}\,(v) , \Omega) >0, \quad \displaystyle\int_{\Omega^{c}}\frac{\vert v(x) \vert^{m-1}}{(1+\vert x \vert)^{N + sm}} \dd x < \infty, \ \mbox{for} \ m \in \{p,q\}.
\end{align*}
		
Then, $u + v \in \mathcal{W}(\Omega)$ satisfy $(-\Delta_p)^s(u+v) + (-\Delta_q)^s(u+v) = f + h$ weakly in $\Omega$, where
\begin{equation*}
h(x) = 2 \displaystyle\sum_{m=p,q} \displaystyle\int_{\textup{supp}\, (v)} \frac{J_m( u(x) - u(y) - v(y)) - J_m( u(x) - u(y))}{\vert x - y \vert^{N+sm}}\dd x
\end{equation*}
\end{lemma}
\begin{proof} It suffices to consider the case when $\Omega$ is bounded. Define $K =\textup{supp}\,(v)$ and consider $U \subset \mathbb{R}^N$ such that, 
\begin{equation*}
\Omega \Subset U \quad \mbox{and} \quad \vert \vert u \vert \vert_{W^{s,m}(U)} + \displaystyle\int_{\mathbb{R}^N} \frac{\vert u(x)\vert^{m-1}}{(1+\vert x \vert)^{N+sm}} \dd x < \infty
\end{equation*}
for $m \in \{p,q\}$.

Without loss of generality we can assume that $\Omega \Subset  U \Subset  K^{c}$, since $\textup{dist}\,(\Omega, K)=d>0$. Clearly $u+v = u$ in $U$, and thus $u+v \in W^{s,m}(U)$ for $m \in \{p,q\}$. Moreover, for $m \in \{p,q\}$ we have
\begin{equation*}
\displaystyle\int_{\mathbb{R}^N} \frac{\vert u(x) + v(x) \vert^{m-1}}{(1+\vert x\vert)^{N+sm}}\dd x \leq C \displaystyle\int_{\mathbb{R}^N} \frac{\vert u(x)\vert^{m-1}}{(1+\vert x \vert)^{N+sm}} \dd x + C\displaystyle\int_{\mathbb{R}^N} \frac{\vert u(x)\vert^{m-1}}{(1+\vert x \vert)^{N+sm}} \dd x < \infty.
\end{equation*} 
		
Therefore, $u+v \in \mathcal{W}(\Omega)$. 
		
Now assume that $(-\Delta_p)^su + (-\Delta_q)^su = f$ weakly in $\Omega$. Choose $\varphi \in C^{\infty}_0(\Omega)$ and compute
\begin{align*}
\sum_{m=p,q}\displaystyle\int_{\mathbb{R}^N}&\displaystyle\int_{\mathbb{R}^N} \frac{J_m(u(x)+v(x)-u(y)-v(y))(\varphi(x) - \varphi(y))}{\vert x-y \vert^{N+sm}} \dd x\dd y\hfill \\
&= \displaystyle\sum_{m=p,q}\displaystyle\int_{\Omega}\displaystyle\int_{\Omega} \frac{J_m(u(x)-u(y))(\varphi(x) - \varphi(y))}{\vert x-y \vert^{N+sm}} \dd x\dd y \\
&\qquad+\displaystyle\sum_{m=p,q}\displaystyle\int_{\Omega^c}\displaystyle\int_{\Omega} \frac{J_m(u(x)-u(y)-v(y))\varphi(x)}{\vert x-y\vert^{N+sm}} \dd x\dd y \\
&\qquad-\displaystyle\sum_{m=p,q}\displaystyle\int_{\Omega}\displaystyle\int_{\Omega^c} \frac{J_m(u(x)-u(y)+v(y))\varphi(y)}{\vert x-y\vert^{N+sm}} \dd x\dd y \\
&=\displaystyle\sum_{m=p,q}\displaystyle\int_{\mathbb{R}^N}\displaystyle\int_{\mathbb{R}^N}\frac{J_m(u(x)-u(y))(\varphi(x) - \varphi(y))}{\vert x-y\vert^{N+sm}}\dd x\dd y \\
&\qquad+ \displaystyle\sum_{m=p,q}\displaystyle\int_{\Omega}\displaystyle\int_{\Omega^c} \frac{J_m(u(x)-u(y))\varphi(y)}{\vert x-y\vert^{N+sm}} \dd x\dd y \\
&\qquad-\displaystyle\sum_{m=p,q}\displaystyle\int_{\Omega^c}\displaystyle\int_{\Omega} \frac{J_m(u(x)-u(y)-v(y))\varphi(x)}{\vert x-y\vert^{N+sm}} \dd x\dd y \\
&\qquad+2\displaystyle\sum_{m=p,q}\displaystyle\int_{\Omega}\displaystyle\int_{\Omega^c} \frac{J_m(u(x)-u(y)-v(y))\varphi(x)}{\vert x-y\vert^{N+sm}} \dd y \dd x \\
&= \displaystyle\int_{\Omega} \left( f(x)  + 2\displaystyle\sum_{m=p,q} \displaystyle\int_{\Omega^c} \frac{J_m(u(x)-u(y)-v(y)) - J_m(u(x)-u(y))}{\vert x-y \vert^{N+sm}}\dd y\right) \varphi(x) \dd x \\
&=\displaystyle\int_{\Omega} (f(x) + h(x)) \varphi(x)\dd x.
\end{align*}
\end{proof}
\begin{proposition}[Comparison Principle]\label{PC}
Let $\Omega$ be bounded, and $u,v \in \mathcal{W}(\Omega)$ satisfy $u \leq v$ in $\Omega^c$. Suppose that, for all $\varphi \in W_0^{s,p}(\Omega) \cap W_0^{s,q}(\Omega)$, $\varphi \geq 0$ in $\Omega$, it is valid
\begin{align*}
\displaystyle\sum_{m=p,q}\displaystyle\int_{\mathbb{R}^N}\displaystyle\int_{\mathbb{R}^N}&\frac{J_m(u(x)-u(y))(\varphi(x) - \varphi(y))}{\vert x-y\vert^{N+sm}}\dd x\dd y  \\
& \leq \displaystyle\sum_{m=p,q}\displaystyle\int_{\mathbb{R}^N}\displaystyle\int_{\mathbb{R}^N}\frac{J_m(v(x)-v(y))(\varphi(x) - \varphi(y))}{\vert x-y\vert^{N+sm}}\dd x\dd y. 
\end{align*}
Then $u \leq v$ in $\Omega$.
\end{proposition}
\begin{proof}
The prove is a straightforward calculus, but for convenience of the reader we sketch the details. We 
Subtracting the above equations and adjusting the terms, yields
\begin{align}\label{int1}
0&\geq \displaystyle\int_{\mathbb{R}^{N}}\displaystyle\int_{\mathbb{R}^{N}} \left(\frac{J_{p}(u(x) - u(y)) - J_{p}(v(x) - v(y)) }{\vert x-y\vert^{N+sp}}\right)(\varphi(x) - \varphi(y)) \dd x\dd y\nonumber\\
&\quad+ \displaystyle\int_{\mathbb{R}^{N}}\displaystyle\int_{\mathbb{R}^{N}} \left(\frac{J_{q}(u(x) - u(y)) - J_{q}(v(x) - v(y))}{\vert x-y\vert^{N+sq}}\right)(\varphi(x)-\varphi(y)) \dd x\dd y,
\end{align}
since $\varphi\geq 0$.

We show that the integrand is non-negative for $\varphi = (u-v)^{+} \in \mathcal{W}$ (See Proposition 2.10 in \cite{Ian}). Taking $a = v(x) - v(y)$ and $b = u(x)- u(y)$ , the identity
\begin{align*}
J_{m}(b) - J_m(a) = (m-1)(b-a)\displaystyle\int_{0}^{1} \vert a + t(b-a) \vert^{m-2} \dd t
\end{align*}
yields
\begin{align*}
J_{m}(u(x) - u(y)) - J_{m}(v(x) - v(y)) = (m-1)\left[(u-v)(x) - (u-v)(y) \right]Q_{m}(x,y),
\end{align*}
where $Q_{m}(x,y) = \displaystyle\int_{0}^{1}\left\vert (v(x) - v(y)) + t[(u-v)(x) - (u-v)(y)]\right\vert^{m-2}\dd t$.

We have $Q_{m}(x,y) \geq 0$ and $Q_{m}(x,y) = 0$ only if $v(x) = v(y)$ and $u(x) = u(y)$.

Rewriting the integrands in \eqref{int1} we obtain
\begin{multline} \label{int2}
\displaystyle\int_{\mathbb{R}^{N}}\displaystyle\int_{\mathbb{R}^{N}} \left(\frac{(p-1)\left[(u-v)(x) - (u-v)(y) \right]Q_{p}(x,y)}{\vert x-y\vert^{N+sp}}\right)(\varphi(x)-\varphi(y))\dd x\dd y\\
+ \displaystyle\int_{\mathbb{R}^{N}}\displaystyle\int_{\mathbb{R}^{N}} \left(\frac{(q-1)\left[(u-v)(x) - (u-v)(y) \right]Q_{q}(x,y)}{\vert x-y\vert^{N+sq}}\right)(\varphi(x)-\varphi(y)) \dd x\dd y  \leq 0. 
\end{multline}
We now choose the test function $\varphi = (u-v)^{+}$ and define
\[
\psi = u-v = (u-v)^{+} - (u-v)^{-} , \quad \varphi = (u-v)^{+} = \psi^{+}.
\]
From (\ref{int2}) results that
\begin{align*}
\displaystyle\int_{\mathbb{R}^{N}}\displaystyle\int_{\mathbb{R}^{N}} &\left(\frac{(p-1)(\psi(x) - \psi(y) )(\psi^{+}(x)-\psi^{+}(y))Q_{p}(x,y)}{\vert x-y\vert^{N+sp}}\right)\dd x\dd y \\
&+\displaystyle\int_{\mathbb{R}^{N}}\displaystyle\int_{\mathbb{R}^{N}} \left(\frac{(q-1)(\psi(x) - \psi(y) )(\psi^{+}(x)-\psi^{+}(y))Q_{q}(x,y)}{\vert x-y\vert^{N+sq}}\right)\dd x\dd y \leq 0
\end{align*}
Using the inequality
\[
(\xi - \eta)(\xi^{+} - \eta^{+}) \geq \vert \xi^{+} - \eta^{+} \vert^{2}
\]
we can see that
\begin{align*}
\displaystyle\int_{\mathbb{R}^{N}}\displaystyle\int_{\mathbb{R}^{N}} &\frac{(p-1)\vert \psi^{+}(x)-\psi^{+}(y)\vert^{2}Q_{p}(x,y)}{\vert x-y\vert^{N+sp}} \dd x\dd y \\
&+ \frac{(q-1)\vert \psi^{+}(x)-\psi^{+}(y)\vert^{2}Q_{q}(x,y)}{\vert x-y\vert^{N+sq}} \dd x\dd y \leq 0.
\end{align*}
Thus
\[
\psi^{+}(x) = \psi^{+}(y) \ \ \mbox{or} \ \ Q_{m}(x,y) = 0,
\]
at a. e. point $(x,y)$. Also the latter alternative implies that $\psi^{+}(x) = \psi^{+}(y)$, and so
\[
(u-v)^{+}(x) = C \geq 0,  \ \ \forall x \in \mathbb{R}^{N}.
\]
The boundary condition implies that $C=0$ and consequently $v\geq u$ in $\mathbb{R}^{N}$.
\end{proof}
\begin{proposition}\label{Prop1}
Suppose $\Omega$ is bounded, $u \in \tilde{W}^{s,m}(\Omega) \cap C_{loc}^{1,\gamma}(\Omega)$, with $\gamma \in [0,1]$ such that
$$ \gamma > \begin{cases}
1 - m(1-s),&\mbox{if}\quad m\geq 2,\\
\frac{1-mp(1-s)}{m-1}, &\mbox{if} \quad m<2.
\end{cases}
$$
Then $(-\Delta_m)^su = f$ strongly in $\Omega$ for some $f \in L_{loc}^{\infty}(\Omega)$.
\end{proposition}
\begin{proof}
See Proposition 2.12, \cite{Ian}. 
\end{proof}
\section{Interior Holder regularity}\label{SEC}
	
Now we assume that $2\leq q \leq p < \infty$ and we will prove a weak Harnack type inequality for non-negative supersolutions and then we will obtain an estimative of the oscillation of a bounded weak solution in a ball. Denote by $B_{R} = B(0;R)$ and we will continue with the notation $\mathcal{W}(\Omega) = \tilde{W}^{s,p}(\Omega) \cap \tilde{W}^{s,q}(\Omega)$. 
\begin{theorem}\label{DH1}
Let $2\leq q \leq p <\infty$ and $u \in \mathcal{W}(B_{R/3})$ satisfying weakly 
\begin{empheq}[left=\empheqlbrace]{align}
(- \Delta_{p})^{s}u + (- \Delta_{q})^{s}u &\geq -K \quad \mbox{em} \quad B_{R/3} \\ 
u(x) &\geq 0, \quad  x \in \mathbb{R}^{N} \nonumber 
\end{empheq}
for some $K\geq 0$. Then exist $\sigma \in (0,1)$ and $\overline{C} >0$ such that
\begin{equation*}
\displaystyle\inf_{B_{R/4}} u \geq \sigma\left(\displaystyle\dashint_{B_{R} \backslash B_{R/2}} u^{q-1} \dd x \right)^{1/(q-1)} - \overline{C}\left(KR^{sp}\right)^{1/(p-1)}.
\end{equation*}
\end{theorem}
\begin{proof}
Choose a function $\varphi \in C^{\infty}(\mathbb{R}^{N})$ be such that $0\leq \varphi \leq 1$ in $\mathbb{R}^N$, $\varphi \equiv 1$ in $B_{3/4}$ and $\varphi \equiv 0$ in $B_{1}^c$. By Proposition \ref{Prop1},  $\vert (-\Delta_m)^s \varphi \vert \leq C_1$ weakly in $B_1$, for $m \geq 2$. We rescale by setting $\varphi_R(x) = \varphi(3x/R),$ so $\varphi_{R} \in C^{\infty}(\mathbb{R}^N)$, $0 \leq  \varphi_{R} \leq 1$ in $\mathbb{R}^N$,\,  $\varphi_R \equiv 1$ in  $B_{R/4},\, \varphi_R \equiv 0$ in $B_{R/3}^c$  and $\vert (-\Delta_m)^s \varphi \vert \leq C_1R^{-sm}$ weakly in $B_{R/3}$.
Given $\sigma \in (0,1)$, consider
\begin{equation*}
L(m) = \left(\displaystyle\dashint_{B_{R} \backslash B_{R/2}}u^{m-1}\dd x \right)^{1/(m-1)} \ \mbox{and} \ w= \sigma L(q)\varphi_{R} + \chi_{B_R \backslash B_{R/2}}u.
\end{equation*}
Thus $w \in \mathcal{W}(B_{R/3})$, and by Lemma \ref{lem1} we have weakly in $B_{R/3}$,
\begin{align*}
(-\Delta_p&)^sw(x) + (-\Delta_q)^sw(x) = (-\Delta_p)^s(\sigma L(q) \varphi_{R}(x)) + (-\Delta_p)^s(\sigma L(q) \varphi_R(x))\\
&\,\,\, + 2 \displaystyle\sum_{m=p,q} \displaystyle\int_{B_{R} \backslash B_{R/2}}\frac{J_m\big(\sigma L(q) \varphi_{R}(x) - u(y)\big)- J_m\big(\sigma L(q) \varphi_{R}(x)\big)}{\vert x-y \vert^{N+sm}} \dd y\\
&\leq (\sigma L(q))^{p-1}(-\Delta_p)^s \varphi_{R}(x)+ (\sigma L(q))^{q-1}(-\Delta_q)^s \varphi_{R} \\
&\,\,\,+ 2 \displaystyle\sum_{m=p,q} \displaystyle\int_{B_{R} \backslash B_{R/2}}\frac{J_m\big(\sigma L(q) \varphi_{R}(x) - u(y)\big)- J_m\big(\sigma L(q) \varphi_{R}(x)\big)}{\vert x-y \vert^{N+sm}}\dd y 
\end{align*}
Thus using the inequality \eqref{inq5} results
\begin{align*}
(-\Delta_p&)^sw(x) + (-\Delta_q)^sw(x)
 \leq \frac{C_1(\sigma L(q))^{p-1}}{R^{sp}} + \frac{C_1(\sigma L(q))^{q-1}}{R^{sq}}\\
 &\,\,\ - 2^{3-p} \displaystyle\int_{B_{R} \backslash B_{R/2}} \frac{(u(y))^{p-1}}{\vert x-y \vert^{N+sp}} \dd y - 2^{3-q} \displaystyle\int_{B_{R} \backslash B_{R/2}} \frac{(u(y))^{q-1}}{\vert x-y \vert^{N+sq}} \dd y \\
& \leq \frac{C_1(\sigma L(q))^{p-1}}{R^{sp}} + \frac{C_1(\sigma L(q))^{q-1}}{R^{sq}} - \frac{C_2( L(p))^{p-1}}{R^{sp}} - \frac{C_1(L(q))^{q-1}}{R^{sq}}
\end{align*}
Applying the Holder inequality we have $L(q) \leq L(p)$, for $p\geq q$, thus, since $\sigma \in (0,1)$ result of the inequality above that
\begin{equation*}
(-\Delta_p)^sw(x) + (-\Delta_q)^sw(x) \leq \left(C_1 \sigma^{q-1} - C_2 \right) \left( \frac{(L(q))^{p-1}}{R^{sp}} +  \frac{(L(q))^{q-1}}{R^{sq}} \right) 
\end{equation*}
		
If choose $0<\sigma < \min\left\{1,\left(\frac{C_2}{2C_1} \right)^{1/(q-1)} \right\}$ we get the upper estimate
\begin{align}
(-\Delta_p)^sw(x) + (-\Delta_q)^sw(x) \leq -\frac{C_2}{2} \frac{(L(q))^{p-1}}{R^{sp}}. \label{opt}
\end{align}

We set $\overline{C}=\left(\frac{2}{C_2}\right)^{1/(p-1)}$ and distinguish two cases:

\noindent $\bullet$  If $L(q) \leq \overline{C}(KR^{sp})^{1/(p-1)}$, then
\begin{equation*}
\displaystyle\inf_{B_{R/4}}u \geq 0 \geq \sigma L(q) - \overline{C}(KR^{sp})^{1/(p-1)};
\end{equation*}
$\bullet$ If $L(q) > \overline{C}(KR^{sp})^{1/(p-1)}$ then using \eqref{opt} we obtain weakly in $B_{R/3}$
\begin{empheq}[left=\empheqlbrace]{align}
(- \Delta_{p})^{s}w(x) + (- \Delta_{q})^{s}w(x) &\leq -K \leq (- \Delta_{p})^{s}u(x) + (- \Delta_{q})^{s}u(x) \\ 
w = \chi_{B_{R}\backslash B_{R/2}}u \,\, &\geq \,\,\, u , \,\,\,  x \in B_{R/3}^c.
\end{empheq}
Using the Proposition \ref{PC}, we  obtain that $w \leq u$ in $\mathbb{R}^N$, in particular
\begin{equation*}
\displaystyle\inf_{B_{R/4}}u \geq \displaystyle\inf_{B_{R/4}}w \geq \sigma L(q)\displaystyle\inf_{B_{R/4}}\varphi_{R} = \sigma L(q) \geq \sigma L(q) - \overline{C}(KR^{sp})^{1/(p-1)}.
\end{equation*}
\end{proof}
\begin{lemma}\label{DH2}
Let $R <1$, $2\leq q \leq p < \infty$ and $u \in \mathcal{W}(B_{R/3})$  such that
\begin{empheq}[left=\empheqlbrace]{align}
(- \Delta_{p})^{s}u + (- \Delta_{q})^{s}u &\geq -K \ \ \mbox{in} \ \ B_{R/3} \\ 
u &\geq 0, \ \mbox{in} \ B_R \nonumber,
\end{empheq}
for some $K\geq 0$. If $u \in L^{\infty}(\mathbb{R}^N)$ then there exists $\sigma \in (0,1)$, $K_0 >0$,  $\overline{C} >0$  and for all $\varepsilon >0$ a constant $C_{\varepsilon}>0$ such that
\begin{align*}
\displaystyle\inf_{B_R} u \geq \sigma \left(\displaystyle\dashint_{B_{R} \backslash B_{R/2}}u^{q-1}\dd x \right)^{\frac{1}{q-1}}- \overline{C}(K_0R^{s(p-q)})^{\frac{1}{p-1}} -\varepsilon \displaystyle\sup_{B_{R}} u - C_{\varepsilon}Tail_{p}(u_{-};R).
\end{align*}
\end{lemma}
\begin{proof}
Let us apply the Lemma \ref{lem1}  for the functions $u$ and $v = u_{-}$,  so that $u_{+} = u+v$ and $\Omega = B_{R/3}$. Then have in a weak sense in $B_{R/3}$
\begin{align*}
(-\Delta_p)^s &u_{+}(x) + (-\Delta_q)^s u_{+}(x) = (-\Delta_p)^s u(x) + (-\Delta_q)^s u(x) \\
& \,\,\, + 2 \displaystyle\sum_{m=p,q} \displaystyle\int_{B_{R/3}^c} \frac{J_m(u(x)-u(y)-u_{-}(y)) - J_m(u(x) - u(y))}{\vert x-y \vert^{N+sm}} \dd y \\
&\geq -K + C\displaystyle\sum_{m=p,q} \displaystyle\int_{ \{u<0\}} \frac{(u(x))^{m-1} - (u(x)-u(y))^{m-1}}{\vert x-y \vert^{N+sm}} \dd y \\
&\geq -K - C\displaystyle\sum_{m=p,q} \displaystyle\int_{ \{u<0\}} \frac{(u(x)-u(y))^{m-1} - (u(x))^{m-1} }{\vert y \vert^{N+sm}} \dd y 
\end{align*}
where in the end was used that $\vert x-y\vert \geq \frac{2}{3} \vert y \vert$, for all $y \in \{u<0 \} \subset B_{R}^c$ and $x \in B_{R/3}$. By inequality \eqref{inq4}, for any $\theta >0$ exists $C_{\theta}>0$ such that weakly in $B_{R/3}$
\begin{align*}
\big((-\Delta_p)^s &+ (-\Delta_q)^s \big) u_{+}(x) \geq -K\!\! - C\!\!\displaystyle\sum_{m=p,q} \displaystyle\int_{ \{u<0\}}\!\!\frac{ \theta (u(x))^{m-1} - C_{\theta}(u(y))^{m-1}}{ \vert y \vert^{N+sm}} \dd y \\
&\geq -K - C \theta\!\! \displaystyle\sum_{m=p,q} (\displaystyle\sup_{B_R})^{m-1} \displaystyle\int_{B_R^c}\!\!\frac{1}{\vert y \vert^{N+sm}} \dd y - C_{\theta}\!\!\displaystyle\sum_{m=p,q}\displaystyle\int_{B_R^c}\!\!\frac{(u(y))^{m-1}}{ \vert y \vert^{N+sm}} \dd y \\
&\geq -K - \frac{C\theta}{R^{sp}} \left(\displaystyle\sup_{B_R}u\right)^{p-1} - \frac{C_{\theta}}{R^{sp}} \left(Tail_{p}(u_{-};R)\right)^{p-1} \\
&\,\,\,\,\,- \frac{C\theta}{R^{sq}} \left(\displaystyle\sup_{B_R}u\right)^{q-1}  - \frac{C_{\theta}}{R^{sq}} \left(Tail_{q}(u_{-};R)\right)^{q-1} \\
&: = -\tilde{K}.
\end{align*}
Using the Remark \ref{OBSFUN} we can see that $Tail_q(u_{-};R) \leq C_0$, where $C_0$ is independent of $R>0$, we also have that $R^{sp} \leq R^{s(p-q)}$ for $R\in(0,1]$, since $q\leq p$.
Thus,
\begin{align*}
\tilde{K}&R^{sp} \leq KR^{sp} + C \theta \left(\displaystyle\sup_{B_{R}} u\right)^{p-1} + C_{\theta} \left(Tail_{p}(u_{-};R)\right)^{p-1} \\
&\,\,\,\, + C R^{s(p-q)} \theta\left(\displaystyle\sup_{B_{R}} u\right)^{q-1}  + C_{\theta}R^{s(p-q)}\left(Tail_{q}(u_{-};R))^{\frac{q-1}{p-1}}\right)^{q-1} \\
&\leq KR^{s(p-q)} + C \theta \left(\displaystyle\sup_{B_{R}} u\right)^{p-1} \!\!+ C_{\theta} \left(Tail_{p}(u_{-};R)\right)^{p-1} \!\!+ \left(C \theta + C_{\theta}\right) M_0R^{s(p-q)}.
\end{align*}
where $M_0 >0$ is a constant independent of $R>0$, that depend of $\|u\|_{L^{\infty}(\mathbb{R}^{N})}$.
		
Consequently, given $\varepsilon > 0$  take $\theta < \min\left\{ 1, \frac{\varepsilon}{C^{p-1}} \right\}$ we have
\begin{align*}
\left(\tilde{K}R^{sp}\right)^{\frac{1}{p-1}} \leq (K_0R^{s(p-q)})^{\frac{1}{p-q}} + \varepsilon \displaystyle\sup_{B_R} u + C_{\varepsilon} Tail_{p}(u_{-};R)
\end{align*}
where $K_0= K_0(K,\|u\|_{L^{\infty}(\mathbb{R}^{N})})>0$ is independent of $R>0$.
		
Therefore, applying the Lemma \ref{DH1} for $u_+$ results
\begin{align*}
\displaystyle\inf_{B_{R/4}} &= \displaystyle\inf_{B_{R/4}} u_{+}\geq \sigma \left(\displaystyle\dashint_{B_{R} \backslash B_{R/2}}u^{q-1}\dd x \right)^{1/(q-1)} - (\tilde{K}R^{sp})^{\frac{1}{p-1}} \\
&\geq \sigma \left(\displaystyle\dashint_{B_{R} \backslash B_{R/2}}u^{q-1}\dd x \right)^{\frac{1}{q-1}}\!\! - (K_0R^{s(p-q)})^{\frac{1}{p-1}} - \varepsilon \displaystyle\sup_{B_R} u - C_{\varepsilon} Tail_{p}(u_{-};R).
\end{align*}
\qed
\end{proof}
	
Now we use the above results to produce an estimate of the oscillation of a bounded function
$u$ such that $(-\Delta_p)^su + (-\Delta_q)^su$ is locally bounded. We set for all $R>0$, $x_0 \in \mathbb{R}^N$
\[
Q(u;x_0;R) = \vert \vert u \vert \vert_{L^{\infty}(B_{R}(x_0)} + Tail_{p}(u;x_0;R), \ \ \ Q(u;R) = Q(u;0;R).
\]
	
\begin{theorem}\label{oscil}
Let $u \in \mathcal{W}(B_{R_0}) \cap L^{\infty}(\mathbb{R}^{N})$ a function such that for some $K\geq 0$ and $R_0 \in (0,1]$ we have 
\begin{equation*}
\vert (-\Delta_p)^s u + (-\Delta_q)^su \vert \leq K \,\, \mbox{weakly in} \,\, B_{R_0}.
\end{equation*}
Then exists $\alpha \in (0,1)$ and $C>0$ such that
\begin{align*}
\underset{B_r}{\osc}u \leq C \left[\left(K_0R_0^{s(p-q)}\right)^{\frac{1}{p-1}} + Q(u;r_0)\right] \left(\frac{r}{R_0}\right)^{\alpha} 
\end{align*}where $K_0 >0$ is given in the Lemma \ref{DH2}.
\end{theorem}
\begin{proof}
Recall that $2\leq q \leq p <\infty$. For all integer $j \geq 0$  we set $R_j = \frac{R_0}{4^j}, B_j = B_{r_j}$ and $\frac{1}{2} B_j = B_{R_j/2}$. We claim that there are $\alpha \in (0,1)$ and $\lambda>0$, a non-decreasing sequence $(m_j)$ and a non-increasing sequence $(M_j)$, such that
\begin{align*}
m_j \leq \displaystyle\inf_{B_j}u \leq \displaystyle\sup_{B_j} \leq M_j, \  M_j - m_j = \lambda R_{j}^{\alpha}, \quad \mbox{for any} \quad j \geq 0.
\end{align*}		
We argue by induction on $j$. \\
Step zero: We set $M_0 = \displaystyle\sup_{B_{R_0}} u$ and $m_0=M_0-\lambda R_j^{\alpha}$, where $0 < \lambda < \frac{2 \vert \vert u \vert \vert_{L^\infty(B_{R_0})}}{R_0^{\alpha}}$.
		
Inductive step: Assume that sequences $(m_j)$ and $(M_j)$ are constructed up to the index $j$. Then
\begin{align*}
M_j - m_j &= \displaystyle\dashint_{B_{R} \backslash B_{R/2}}(M_j-u)\dd x + \displaystyle\dashint_{B_{R} \backslash B_{R/2}}(u-m_j)\dd x  \\
&\leq \left(\displaystyle\dashint_{B_{R} \backslash B_{R/2}}(M_j - u)^{q-1}\dd x\right)^{\frac{1}{q-1}} + \left(\displaystyle\dashint_{B_{R} \backslash B_{R/2}}(u - m_j)^{q-1}\dd x\right)^{\frac{1}{q-1}}
\end{align*}
Note that as $(M_j)$ is non-increasing and $(m_j)$ is non-decreasing, we have $M_{j} -u$ and $u-m_{j}$ are bounded in $\mathbb{R}^{N}$, moreover for all $j \geq 0$ we have
\[
M_{j} - u  \leq M_0 - u \ \ \mbox{and} \ \ u-m_{j} \leq u-m_0.
\]
Let $\sigma \in (0,1)$, $\tilde{C} > 0$ be as in Lemma \ref{DH2}, and multiply the previous inequality by $\sigma$ to obtain, via Lemma \ref{DH2}
\begin{align*}
\sigma (M_j-m_j)
&\leq \displaystyle\inf_{B_{j+1}}(M_j-u) + \displaystyle\inf_{B_{j+1}}(u-m_j) + 2\overline{C}(K_0R_j^{s(p-q)})^{\frac{1}{p-1}}  \\
& \,\,\,\,+ \varepsilon \left[\displaystyle\sup_{B_j}(M_j-u) + \displaystyle\sup_{B_j}(u-m_j)\right] + C_{\varepsilon}Tail_{p}((M_j-u)_{-};R_j)\\
&\,\,\,\,+ C_{\varepsilon}Tail_{p}((u-m_j)_{-};R_j) 
\end{align*}
Setting universally $\varepsilon = \frac{\sigma}{4}, \ \ C= \max\{2\tilde{C},C_{\varepsilon}\}$ and rearranging, we have
\begin{align*}
\underset{B_{j+1}}{\osc}&u \leq (1-\frac{\sigma}{2})(M_{j} - m_{j}) \\
&+ C\left[\left(K_{0}R_{0}^{s(p-q)}\right)^{\frac{1}{p-1}} + Tail_{p}((M_j-u)_{-};R_j) + Tail_{p}((u-m_j)_{-};R_j)\right]
\end{align*}
In the   proof of the Theorem $5.4$ in \cite{Ian} we provide an estimate of both non-local tails, 
\[
Tail_{p}((M_{j} - u)_{-};R_j)  \leq C\left[\lambda S(\alpha)^{\frac{1}{p-1}} + \frac{Q(u;R_0)}{R_{0}^{\alpha}} \right]R_{j}^{\alpha}, \,\,\, S(\alpha) \rightarrow 0 \,\,\, \mbox{as} \,\, \alpha \rightarrow 0
\]
the same being valid for $Tail_{p}((u-m_{j})_{-})$. 
Therefore
\[
\underset{B_{j+1}}{\osc}u \leq \left(1-\frac{\sigma}{2}\right)(M_{j} - m_{j}) + C\left(K_{0}R_{0}^{s(p-q)}\right)^{\frac{1}{p-1}}\!\! + C\left[\lambda S(\alpha)^{\frac{1}{p-1}} +\!\! \frac{Q(u;R_0)}{R_{0}^{\alpha}} \right]R_{j}^{\alpha}.
\]
Recalling that $M_{j}-m_{j} = \lambda R_{j}^{\alpha}$ and $R_{j} = \frac{R_0}{4^{j}}$, follows that
\begin{align*}
\underset{B_{j+1}}{ \osc}u &\leq 4^{\alpha} \left[\left(1-\frac{\sigma}{2}\right) + CS(\alpha)^{\frac{1}{p-1}} \right]\lambda R_{j+1}^{\alpha} \\
&\,\,\,\,\,\,+ 4^{\alpha}C \left[K_{0}^{\frac{1}{p-1}}R_{j}^{\frac{s(p-q)}{p-1}-\alpha} + \frac{Q(u;R_{0})}{R_{0}^{\alpha}} \right]R_{j+1}^{\alpha}.
\end{align*}		
Now we choose $\alpha \in \left(0,\frac{s(p-q)}{p-1}\right)$ universally such that
\[
4^{\alpha}\left[ \left(1-\frac{\sigma}{2} \right) + CS(\alpha)^{\frac{1}{p-1}} \right] < \left(1-\frac{\sigma}{4}\right)
\]
which is possible because $S(\alpha) \rightarrow 0$ as $\alpha \rightarrow 0$. Now, setting
\begin{align}
\lambda = \frac{4^{\alpha + 1}}{\sigma} C\left[K_{0}^{\frac{1}{p-1}}R_{0}^{\frac{s(p-q)}{p-1} - \alpha} + \frac{Q(u;R_0)}{R_0^{\alpha}} \right]. \label{Estimativa}
\end{align}
we have $\lambda \geq \frac{2\vert \vert u \vert \vert_{L^{\infty}(B_{R_0}(x_0))}}{R_0^{\alpha}}$, since that $4^{\alpha + 1}C/\sigma > 2$ and
\[
\underset{B_{j+1}}{\osc} u \leq (1-\frac{\sigma}{4})\lambda R_{j+1}^{\alpha} + \frac{\sigma}{4}\lambda R_{j+1}^{\alpha}.
\]
We may pick $m_{j+1}$, $M_{j+1}$ such that
\[
m_{j} \leq m_{j+1} \leq \displaystyle\inf_{B_{j+1}} u \leq \displaystyle\sup_{B_{j+1}} u  \leq M_{j+1} \leq M_{j}, \ \ M_{j+1} -m_{j+1} = \lambda R_{j+1}^{\alpha},
\]
which completes the induction and proves the claim.
		
Now fix $r \in (0,R_{0})$ and find an integer $j \geq 0$ such that $R_{j+1} \leq r \leq R_{j}$, thus $R_{j} \leq 4r$. Hence, by the claim and \eqref{Estimativa}, we have
\[
\underset{B_r}{\osc} \leq \underset{B_{j}}{\osc} \leq \lambda R_{j}^{\alpha} \leq C\left[(K_{0}R_{0}^s(p-q))^{\frac{1}{p-1}} + Q(u;R_0)\right] \left(\frac{r}{R_{0}}\right)^{\alpha},
\]
which concludes the argument.
\end{proof}
	
\begin{corollary}\label{CONT}
Let $u \in \mathcal{W}(B_{2R_{0}}(x_0)) \cap L^{\infty}(\mathbb{R}^{N})$ such that for some $K \geq 0$ and $R_{0} \in (0,1]$ we have 
\[
\left\vert (-\Delta_{p})^{s} u + (-\Delta_{q})^{s} u\right\vert \leq K \,\,
\mbox{weakly in} \,\, B_{2R_{0}}(x_0).
\] 
Then there exists $C>0$ and $\alpha \in (0,1)$ such that
\[
\vert u \vert_{C^{0,\alpha}(B_{R_{0}}(x_{0})} \leq C\left[(K_{0}R_{0}^s(p-q))^{\frac{1}{p-1}} + Q(u;x_{0};2R_0)\right] R_{0}^{-\alpha}
\]
\end{corollary}
\begin{proof}
Given $x,y \in B_{R_{0}}(x_{0})$, let $r=\vert x-y \vert \leq R_{0}$. let us apply the Theorem \ref{oscil} to the ball $B_{R_{0}}(x) \subset  B_{2R_{0}}(x_{0})$. Clearly $\left\Vert u \right\Vert_{L^{\infty}(B_{R_{0}}(x))} \leq \left\Vert u \right\Vert_{L^{\infty}(B_{2R_{0}}(x_0))}$ and
\begin{align*}
Tail_{p}(u;x;R_{0})^{p-1} &= R_{0}^{sp}\displaystyle\int_{B_{R_{0}}^{c}(x)} \frac{\vert u(y) \vert^{p-1}}{\vert x-y \vert^{N+sp}} \dd y \\
&\leq C \left\Vert u \right\Vert_{L^{\infty}(B_{2R_{0}}(x_0))}^{p-1} + CR_{0}^{sp}\displaystyle\int_{B_{2R_{0}}^{c}(x_{0})} \frac{\vert u(y) \vert^{p-1}}{\vert x_{0}-y \vert^{N+sp}} \dd y
\end{align*}
for a universal $C$, where as usual we used $\vert x-y\vert \geq \vert x_{0} - y\vert/2$ for $y \in B_{2R_{0}}^{c}(x_{0})$ and $x \in B_{R_{0}}(x)$. This implies that,
\[
Q(u;x;R_{0}) \leq CQ(u;x_{0};2R_{0})
\]
and thus the desired estimate on the Holder seminorm.
\qed
\end{proof}
	
\section{Boundedness of solutions - a general procedure}\label{bound}
Given an $f \in L^{\frac{p^{*}_{s}}{p^{*}_s-1}}(\mathbb{R}^{N})$, consider the problem
\begin{equation}\label{P_2}\left\{
\begin{array}{rcll}
(- \Delta_{p})^{s}u + (- \Delta_{q})^{s}u =f \,\,\,\, \mbox{in} \,\,\,\, \mathbb{R}^{N} \\ 
u(x)  \geq 0 ,  \,\,\,\, x \in \mathbb{R}^{N}  \end{array} \right.
\end{equation}
where $0 < s < 1$, $N > sp$, $p^{*}_{s} = \frac{Np}{N-sp}$ and $1 < q \leq p < \infty$. We will denote by
\begin{equation*}
\mathcal{W} := D^{s,p}(\mathbb{R}^{N}) \cap D^{s,q}(\mathbb{R}^{N})
\end{equation*}
which is a Banach space with the induced norm
\begin{equation*}
\|u\|_{\mathcal{W}} = \|u\|_{s,p} + \|u\|_{s,q}.
\end{equation*}	
\begin{definition}\label{DEF}
 We say that $u \in \mathcal{W}$ is a weak solution of  \eqref{P_2} if
\begin{align*}
&\displaystyle\int_{\mathbb{R}^N}\!\!\int_{\mathbb{R}^{N}} \!\!\left( \frac{\vert u(x) - u(y) \vert^{p-2}}{\vert x-y\vert^{N+ sp}} + \frac{ \vert u(x) - u(y) \vert^{q-2}}{\vert x-y\vert^{N+ sq}}\right)\!\!(u(x) - u(y))(\varphi(x) - \varphi(y)) \dd x\dd y \\
&= \displaystyle\int_{\mathbb{R}^N} f\varphi \dd x, \ \ \ \mbox{for all} \ \varphi \in \mathcal{W}.
\end{align*}
\end{definition}
The following remark is a direct consequence of the spaces involved and is the key to concluding the continuity of the solutions of \eqref{P_2}.
\begin{remark}\label{REM1}
$1)$ Note that, $\mathcal{W} \subseteq \mathcal{W}(\Omega)$ for any $\Omega \subset \mathbb{R}^{N}$ a bounded domain.\\
$2)$The condition $f \in L^{\frac{p^{*}_{s}}{p^{*}_s-1}}(\mathbb{R}^{N})$ is important so that the functional \linebreak$\varphi \mapsto \displaystyle\int_{\mathbb{R}^{N}} f\varphi \dd x$ is well defined for any 
$\varphi \in \mathcal{W}$.
\end{remark}
	
\textbf{Proof of theorem \ref{BOUNDED}}. By Remark \ref{REM1}, if $u \in \mathcal{W}$ satisfies \eqref{P_2} with $f \in L^{\infty}_{loc}(\mathbb{R}^{N})$, then given $x_0 \in \mathbb{R}^{N}$ we have $u \in \mathcal{W}\left(B_{2R_0}(x_0)\right)$, for any $0<R_0\leq 1$ and $|(-\Delta_{p})^{s}u + (-\Delta_q)^{s} u | \leq K = \|f\|_{L^{\infty}(B_{2R_0}(x_0))}$. Now, by applying Corollary \ref{CONT} results
\begin{equation*}
\vert u \vert_{C^{0,\alpha}(B_{R_{0}}(x_{0})} \leq C\left[(K_{0}R_{0}^{s(p-q)})^{\frac{1}{p-1}} + Q(u;x_{0};2R_0)\right] R_{0}^{-\alpha}.
\end{equation*}

Give $\Omega \subset \mathbb{R}^{N}$ compact, we consider a cover $\Omega \subset \displaystyle\cup_{i} B_{R_i}(x)$ with $x \in \Omega$ and $0 \leq R_i <1$.
We use the same arguments of the proof of Theorem $1.1$ in \cite{Ian}, for conclude that $u \in C^{\alpha}(\Omega)$.

To show that $u \in L^{\infty}(\mathbb{R}^N)$, we assume that $f \in L^{\frac{P^{*}_{s}}{p^{*}_{s}-1}}(\mathbb{R}^{N}) \cap L^{\theta}(\mathbb{R}^{N})$ and we use the Moser iteration process.

Let $M>0$ and $\beta > 1$, we set for simplicity $u_{M} = \min\{u,M \}$ and
\begin{align*}
g_{\beta,M}(t) = (min\{t,M \})^{\beta} =\left\{\begin{array}{rc}
t^{\beta},&\mbox{se} \ t\leq M,\\
M^{\beta}, &\mbox{se} \ t > M,
\end{array}\right.
\end{align*}
we can see that $g_{\beta,M}$ is continuous and has bounded derivative. Hence, 
\[
u_{M} = g_{\beta,M}(u) \in \mathcal{W} \cap L^{\infty}(\mathbb{R}^N).
\]
	
Then we insert the test function $\varphi= g_{\beta,M}(u)$ in the Definition \eqref{DEF} and the Holder inequality, we get
\begin{align}
\displaystyle\int_{\mathbb{R}^N}\!\!\int_{\mathbb{R}^N} &\left(\frac{ J_{p}(u(x) - u(y))}{\vert x-y\vert^{N+ sp }} + \frac{J_{q}(u(x) - u(y))}{\vert x-y\vert^{N+sq }} \right)(g_{\beta,M}(u(x)) - 
g_{\beta,M}(u(y))) \dd x\dd y \nonumber\\
&= \displaystyle\int_{\mathbb{R}^N} f(x)g_{\beta,M}(u(x)) \dd x = \displaystyle\int_{\mathbb{R}^N} f(x)u_M^{\beta}(x) \dd x \leq \Vert f \Vert_{\theta} \Vert u_M^{\beta} \Vert_{\theta'} \label{Des1}
\end{align}	
Setting
\begin{align}
G_{\beta,M}(t) = \displaystyle\int_{0}^{t} (g_{\beta,M}'(\tau))^{\frac{1}{m}} \dd\tau =  \frac{\beta^{\frac{1}{m}}m}{\beta + m -1}(\min\{t,M \})^{\frac{\beta + m-1}{m}} \label{Des2}
\end{align}
and using Lemma \ref{Le} for $m \in \{ p,q \}$, $a = u(x)$, and $b=u(y)$, results of \eqref{Des1} that
\begin{align*}\displaystyle\int_{\mathbb{R}^N}\displaystyle\int_{\mathbb{R}^N} \frac{ \vert G_{ \beta,M}(u(x)) - G_{ \beta,M}(u(y))\vert^{p}}{ \vert x-y \vert^{N+ sp}} \dd x\dd y \leq
  \Vert f  \Vert_{\theta} \Vert u_M^{\beta} \Vert_{\theta'} 
\end{align*}
By Sobolev inequality \eqref{S} we get
\begin{align*}
S\left(\displaystyle\int_{\mathbb{R}^N} \vert G_{\beta,M}(u(x)) \vert^{p_{s}^*} \dd x \right)^{\frac{p}{p_{s}^*}} &\leq \displaystyle
\int_{\mathbb{R}^N}\displaystyle\int_{\mathbb{R}^N} \frac{\vert G_{\beta,M}(u(x)) - G_{\beta,M}(u(y))\vert^{p}}{\vert x-y \vert^{N+ sp}} \dd x\dd y \\
&\leq \vert \vert f \Vert_{\theta} \Vert u_M^{\beta} \Vert_{\theta'}.
\end{align*}
From \eqref{Des2}
\[
S\left(\frac{\beta^{\frac{1}{p}}p}{\beta +p -1}\right)^{p} \left( \displaystyle\int_{\mathbb{R}^N} u_{M}^{\frac{(\beta + p - 1)p_{s}^*}{p}} \dd x \right)^{\frac{p}{p_{s}^*}} \leq \vert \vert f \vert \vert_{\theta}\Vert 
u_M^{\beta} \Vert_{\theta'}.
\]	
equivalently using $\beta >1$
\begin{align}\label{ITM}
\left( \displaystyle\int_{\mathbb{R}^N} u_{M}^{\frac{(\beta + p - 1)p_{s}^*}{p}} \dd x \right)^{\frac{p}{p_{s}^*}} \leq C_1\left(p_{s}^*\frac{\beta +p -1}{p}\right)^{p} \Vert u_M^{\beta} \Vert_{\theta'}
\end{align}
where $C_1= C_1(s,p,N, \Vert f \Vert_{\theta}) >0$.
By setting
\[
\beta_{n+1} = p_{s}^*\frac{\beta_n + p -1}{p\theta'}, \ \beta_0 = \frac{p_{s}^*}{\theta'} >1 \ \mbox{and} \ \sigma_{n} = \frac{\beta_n}{\beta_n + p -1} < 1.
\]
We can see that $\beta_n$ is increasing, and we obtain of \eqref{ITM} for $\beta = \beta_n>1$
\begin{equation}\label{ITM1}
\Vert u_M \Vert_{L^{\theta' \beta_{n+1}}(\mathbb{R}^N)} \leq C^{ \frac{1}{\beta_{n+1}}} \beta_{n+1}^{\frac{\beta_0}{\beta_{n+1}}} \Vert u_M \Vert_{L^{\theta'\beta_n}(\mathbb{R}^N)}^{\sigma_n}.
\end{equation}
Iterating this inequality and using that $\sigma_n < 1$, we get for any $n\geq 1$
\[
\Vert u_M \Vert_{L^{\theta' \beta_{n+1}}(\mathbb{R}^N)} \leq C^{\displaystyle\sum_{j=1}^{n+1}\frac{1}{\beta_j}}\left( \displaystyle\prod_{j=1}^{n+1}\beta_j^{\frac{1}{\beta_j}}\right)^{\beta_0} \Vert u_M \Vert_{L^{p_{s}^*}(\mathbb{R}^N)}^{\displaystyle\prod_{j=0}^{n}\sigma_j}
\]
Setting $\gamma = \frac{p_{s}^*}{\theta'p} = \frac{N(\theta-1)}{(N - sp)\theta}$, we have $\gamma > 1$ since that $\theta > \frac{N}{sp}$, and so
\[
\beta_n = \gamma^n \beta_0 + (p-1)\frac{\gamma^{n+1} - \gamma}{\gamma -1}.
\]
It yields
\[
\displaystyle\lim_{n \rightarrow \infty}\frac{\beta_n}{\gamma^n} = \beta_0 + (p-1)\displaystyle\lim_{n\rightarrow \infty}\frac{\gamma^{n-1}-\gamma}{\gamma^n(\gamma-1)} = \beta_0 + (p-1)\frac{\gamma}{\gamma-1} = \frac{p_{s}^*(p_{s}^* - \theta')}{\theta'(p_{s}^* - \theta'p)}.
\]
Therefore, using the limit comparison test, we conclude that
\[
\sum_{j=1}^{\infty} \frac{1}{\beta_j} <\infty,\ \ \mbox{and}\ \ \prod_{j=1}^{+\infty}\beta_{j}^{\frac{1}{\beta_j}} < \infty.
\]	
Moreover,
\[
\displaystyle\lim_{n\rightarrow \infty} \displaystyle\prod_{j=0}^{n} \sigma_{j} = \displaystyle\lim_{n\rightarrow \infty}\displaystyle\prod_{j=0}^{n} \gamma \frac{\beta_{j}}{\beta_{j+1}} = \displaystyle\lim_{n\rightarrow \infty}\gamma^{n+1} \frac{\beta_0}{\beta_{n+1}} = \displaystyle\lim_{n\rightarrow \infty}\frac{\gamma^{n+1}}{\beta_{n+1}}\frac{p_{s}^*}{\theta'} = \frac{p_{s}^* - \theta'p}{p_{s}^*- \theta'}.
\]
Using these estimates and taking $n \rightarrow +\infty$ in \eqref{ITM1} results
\begin{align}
\Vert u_M \Vert_{L^{\infty}(\mathbb{R}^N)} \leq C \Vert u_M \Vert_{L^{p_{s}^*}(\mathbb{R}^N)}^{\frac{p_{s}^* - \theta'p}{p_{s}^*- \theta'}} \leq C \Vert u \Vert_{L^{p_{s}^*}(\mathbb{R}^N)}^{\frac{p_{s}^* - \theta'p}{p_{s}^*- \theta'}} 
\end{align}
for some $C= C(s,p,N, \vert \vert f \vert \vert_{\theta}) > 0$. We now let $M \rightarrow +\infty$, which gives $u \in L^{\infty}(\mathbb{R}^N)$, and  we get
\[
\vert \vert u\vert \vert_{L^{\infty}(\mathbb{R}^N)} \leq C \vert \vert u \vert \vert_{L^{p_{s}^*}(\mathbb{R}^N)}^{\frac{p_{s}^* - \theta'p}{p_{s}^*- \theta'}}.
\].
\qed 

\begin{remark}
Note that, the condition $2 \leq q \leq p$ is necessary only to prove the continuity of $u$. To prove that $u$ is bounded we can assume that $1 \leq q \leq p$.
\end{remark}
\section{Existence of solution for an problem}\label{SEC1}
Let $\ 1< m < \frac{N}{s}$ and measurable $u : \mathbb{R}^N \rightarrow \mathbb{R}$ the quantity 
\[
\Vert u \Vert_{s,m} = \left(\displaystyle\int_{\mathbb{R}^N} \displaystyle\int_{\mathbb{R}^{N}} \frac{\vert u(x) - u(y) \vert^{m}}{\vert x-y\vert^{N+ s m}} \dd x\dd y \right)^{\frac{1}{m}}.
\]
defines a uniformly convex norm on the reflexive Banach space
\[
D^{s,m}(\mathbb{R}^{N}) = \{u \in L^{m_{s}^*}(\mathbb{R}^N) ; \Vert u \Vert_{s,m} < \infty \} \ \mbox{with} \ m_{s}^* = \frac{Nm}{N - sm}.
\]
	
Let $\mathcal{W}:= D^{s,p}(\mathbb{R}^{N}) \cap D^{s,q}(\mathbb{R}^{N})$, endowed with the norm
\begin{equation*}
\Vert u \Vert_{\mathcal{W}} := \Vert u \Vert_{s,p} + \Vert u \Vert_{s,q}.
\end{equation*}
	
To simplify the notation, we will use $S:=S_{s,p} $ the Sobolev constant.
The following lemma can be found in \cite[Lemma 4.8]{Kavian}
\begin{lemma}\label{Kavian}
Let $\Omega \subset \mathbb{R}^N$ , $1<p<\infty$ and $\{u_n \} \subset L^p(\Omega) $ be a bounded sequence converging to $u$ almost everywhere in $\Omega$. Then $u_n \rightharpoonup u$ in $L^{p}(\Omega)$.
\end{lemma}
	
Next we demonstrates a result related to compactness:
\begin{lemma}\label{CQTP}
Let $(u_n)_{n \in \mathbb{N}}$ be a bounded sequence in $\mathcal{W}$. Then there is $u \in \mathcal{W}$ such that less than subsequence $u_n(x) \rightarrow u(x)$ q.t.p. in $\mathbb{R}^N$. Moreover for $m \in \{p,q \}$ we have
\[
\displaystyle\lim_{n \rightarrow \infty} \Vert u_n - u \Vert_{s,m}^{m} = \displaystyle\lim_{n \rightarrow \infty } \left( \Vert u_n \Vert_{s,m}^{m} - \Vert u \Vert_{ s,m}^{m} \right).
\] 
\end{lemma}
\begin{proof} Let $(u_n)_{n \in \mathbb{N}}$ a sequence in $ \mathcal{W}$ such that,
\begin{align}
\Vert u_n \Vert_{\mathcal{W}} = \Vert u_n \Vert_{s,p} + \Vert u_n \Vert_{s,q} \leq C, \ \ \forall n \in \mathbb{N}.
\end{align}
		
It is easy to see that $\mathcal{W}$ is a uniformly convex Banach space, and hence $\mathcal{W}$ is reflexive Banach space, so there is $u \in \mathcal{W}$ such that $u_n \rightharpoonup u$ in  $\mathcal{W}$.
		
On the other hand, given $\Omega_0 \subset \mathbb{R}^N$ compact, using Holder inequality we have
\begin{align*}
\int_{\Omega_0} \vert u_n \vert^p \dd x + \displaystyle\int_{\Omega_0} \displaystyle\int_{\Omega_0} \frac{\vert u_n(x) - u_n(y) \vert^{p}}{\vert x-y\vert^{N+ sp}} \dd x\dd y &\leq \vert \Omega_0 \vert^{ \frac{N}{sp}} \| u_n \|^{p}_{p^{*}_{s}} + \Vert u_n \Vert_{s,p}^p \\
& \leq \left(\frac{\vert \Omega_0 \vert^{\frac{N}{sp}}}{S} +1\right) \Vert u_n \Vert_{s,p}^p \leq C.
\end{align*} 
		
Therefore $u_n \in W^{s,p}(\Omega_0)$ for each $n \in \mathbb{N}$ and all $\Omega_0$ compact. Since  the embedding $W^{s,p}(\Omega_0) \hookrightarrow L^{p}(\Omega_0)$ is compact, it follows that embedding $\mathcal{W} \hookrightarrow L^{p}_{loc}(\mathbb{R}^N)$ is compact. Hence less than subsequence, $u_n \rightarrow u$ in $L^{p}_{loc}(\mathbb{R}^N)$ and consequently
$u_n(x) \rightarrow u(x)$ q.t.p. in $ \mathbb{R}^N.$ 
		
For the second part of the lemma, let $m \in \{p,q\}$ and defined 
\[
\mathcal{U}_{n}(x,y) = \frac{u_n(x) - u_n(y)}{\vert x-y\vert^{\frac{N}{p} + s}} \in L^{m} \left(\mathbb{R}^N \times \mathbb{R}^N \right)
\]

By the first part obtain
\[ 
\mathcal{U}_{n}(x,y) \rightarrow \mathcal{U}(x,y) = \frac{u(x) - u(y)}{\vert x-y\vert^{\frac{N}{p} + s}}, \,\,\, \mbox{a.e. in} \,\,\, \mathbb{R}^N \times \mathbb{R}^N.
\]
		
Since $(u_n)$ is bounded in $\mathcal{W}$ follow that $\left(\mathcal{U}_{n}\right)_{n \in \mathbb{N}}$ is bounded in $L^{m}(\mathbb{R}^N \times \mathbb{R}^N)$, from the Lemma \ref{Kavian} results 
\[
\mathcal{U}_{n} \rightharpoonup \mathcal{U} \ \mbox{in} \ L^{m}(\mathbb{R}^N \times \mathbb{R}^N).
\]
applying the lemma of Brezis Lieb we complete the proof.
\qed
\end{proof} 

Let us introduce the following version of the mountain pass theorem (see \cite{Springer,Ambroseti,Servadei,Willen}).
\begin{lemma}\label{PM}
Let $X$ be a real Banach space and $\Phi \in C^1(X, \mathbb{R})$. Suppose that $\Phi(0) =0$ an that there exist $\beta, \rho > 0$ and $x_1 \in X\backslash \overline{B}_{\rho}(0)$ such  that
\begin{enumerate}
\item[($i$)] $\Phi(u) \geq \beta$ for all $u \in X$ with $\Vert u \Vert_{X} = \rho$;
\item[($ii$)] $\Phi(x_1) < \beta$.
\end{enumerate}

There exists a sequence $\{u_n\} \subset X$ satisfying 
\begin{equation*}
\Phi(u_n) \rightarrow c \,\,\,\, \mbox{and} \,\,\,\, \Phi'(u_n) \rightarrow 0,
\end{equation*}
where $c$ is the minimax level, defined by
\begin{equation*}
c:= \inf\left\lbrace \displaystyle\max_{t\geq0}\Phi(\gamma(t)) : \gamma \in C([0,1],\mathbb{R}), \gamma(0)=0 \,\,\,\, \mbox{and} \,\,\,\ \gamma(1) = x_1 \right\rbrace.
\end{equation*}
\end{lemma}
We are interested first in finding nontrivial weak solutions to the following problem
\begin{equation}\label{EQ1}
\left\{\begin{array}
[c]{clll}%
(- \Delta_{p})^{s}u + (- \Delta_{q})^{s}u  =  \vert u \vert^{p_{s}^* -2}u + \lambda g(x) \vert u\vert^{r-2}u \,\,\,
\text{in } \,\,\, \mathbb{R}^{N} \\
u(x)  \geq  0  \,\,\,\, x \in \mathbb{R}^{N}.
\end{array}
\right.  
\end{equation}
where $1<q\leq p$, $N > s p $, $\lambda > 0$ is a parameter. The function $g: \mathbb{R}^{N} \rightarrow \mathbb{R}$ satisfying the conditions:
\begin{enumerate}
\item[($g_1$)] $g$ is integrable and $g \in L^{t}( \mathbb{R}^N )$, with $t = \frac{p^{*}_{ s}}{p^{*}_{s} - r}$;
\item[($g_2$)] There exist an open set $\Omega_{g} \in \mathbb{R}^N$ and $\alpha_0 >0$ such that\\ $g(x) \geq \alpha_0 > 0$, for all $x \in \Omega_{g}$.
\end{enumerate}
	
\begin{definition}\label{defsol}
We say that $u \in \mathcal{W}$ is a weak solution of problem \eqref{EQ1} if
\begin{align*}
\int_{\mathbb{R}^N} \displaystyle\int_{\mathbb{R}^{N}} &\left( \frac{J_{p}(u(x) - u(y))}{\vert x-y\vert^{N+ sp}} + \frac{ J_{q}(u(x) - u(y))}{\vert x-y\vert^{N+ sq}}\right)(\varphi(x) - \varphi(y)) \dd x\dd y \\
&= \displaystyle\int_{\mathbb{R}^N} ( u^+ )^{p_s^* -2}u^+ \varphi \dd x + \lambda \displaystyle\int_{\mathbb{R}^N} g(u^+)^{r-2}u^+ \varphi \dd x, \ \ \ \mbox{for all} \ \varphi \in \mathcal{W}.
\end{align*}
\end{definition}
	
Observe that Definition \ref{defsol} is satisfied by critical points of the functional,
\begin{align}
I_{\lambda}(u) = \frac{1}{p} \Vert u\Vert_{s,p}^p + \frac{1}{q}\Vert u \Vert_{s,q}^{q} - \frac{1}{p_{s}^*} \displaystyle\int_{\mathbb{R}^N} (u^+)^{p_{s}^*} \dd x - \frac{\lambda}{r} \displaystyle\int_{\mathbb{R}^N} g(u^+)^r \dd x.\label{Funcional}
\end{align}
where $u^{\pm} = \max \{\pm u,0 \}$.
	
\begin{lemma}
Let $(g_1)$ hold. Then $I_{\lambda}$ is well defined, for all $\lambda > 0$, $I_{\lambda} \in C^1(\mathcal{W}, \mathbb{R})$ and for all $u,\varphi \in \mathcal{W}$ we have
\begin{align}
I_{\lambda}(u)\varphi = &\displaystyle\int_{\mathbb{R}^N} \displaystyle\int_{\mathbb{R}^{N}} \left( \frac{J_{p}(u(x) - u(y))}{\vert x-y\vert^{N+ sp}} + \frac{ J_{q}(u(x) - u(y))}{\vert x-y\vert^{N+ sq}}\right)(\varphi(x) - \varphi(y)) \dd x\dd y \nonumber \\
&= \displaystyle\int_{\mathbb{R}^N} ( u^+ )^{p_s^*-2}u^+ \varphi \dd x + \lambda \displaystyle\int_{\mathbb{R}^N} g(u^+)^{r-2}u^+ \varphi \dd x. \label{FUN}
\end{align}
\end{lemma}
\begin{proof} The proof of this fact is well known. See Lemma 2 in \cite{Pucci}.
\qed
\end{proof} 
	
It is standard to show that the functional $I_\lambda$ has the mountain pass structure on the space $\mathcal{W}$. Thus, for each $\lambda > 0$, the minimax level denoted by
\begin{equation}\label{CL}
\overline{c}_{\lambda} := \displaystyle\inf_{u \in \mathcal{W}\backslash \{0\}}\displaystyle\max_{t\geq0}I_{\lambda}(tu),
\end{equation}
is positive, and there exists a Palais-Smale (PS) sequence  $\{u_n\}\subset \mathcal{W}$ at the level $\overline{c}_{\lambda}$, that is
\begin{align}\label{ps}
I_\lambda(u_n)\rightarrow \overline{c}_{\lambda} \ \ \ \mbox{and}\ \ \ \ I'_\lambda(u_n)\rightarrow0.
\end{align}

\begin{lemma}\label{LIM}
Let $\{u_n\} \subset \mathcal{W}$ be a Palais-Smale sequence. Then $\{u_n\}_{n \in \mathbb{N}}$ is bounded in $\mathcal{W}$.
\end{lemma}
\begin{proof} 
The argument is standard. We indicate the main step. Let $\{u_n\}_{n \in \mathbb{N}}$ such that,
\begin{equation*}
I_{\lambda}(u_n) \leq d_0 \,\, \text{and} \,\, I_{ \lambda}'(u_n) \rightarrow 0 \,\, \text{in} \,\, \mathcal{W}^*.
\end{equation*}	
Thus, for all $n$ large 
\begin{align*}
d_0 +  \Vert u_n \Vert_{\mathcal{W}} &\geq I_{\lambda}(u_n) - \frac{1}{r} I_{\lambda}'(u_n) 
 \geq \left( \frac{1}{p} -\frac{1}{r} \right) \left( \Vert u_n \Vert_{s,p}^p + \Vert u_n \Vert_{s,q}^q \right).
\end{align*}		
From where, we easily conclude that $\Vert u_n \Vert_{\mathcal{W}}$ is bounded.
\qed
\end{proof}

\begin{lemma}\label{C}
Assume that $1<q\leq p < r < p^{*}_{s}$, $(g_1)$ and $(g_2)$ holds. Then exist $\lambda^* >0$ such that $0< \overline{c}_{\lambda} < \frac{s}{N}S^{\frac{n}{sp}}$ for all $\lambda > \lambda^*$.
\end{lemma}
\begin{proof}
By the above comments we have $\overline{c}_{\lambda} >0$.
		
We recall that $\Omega_{g} = \{x \in \mathbb{R}^N ; g(x)\geq \alpha_{0} > 0 \}$.\\
Let $u_0 \in \mathcal{W} \backslash \{0 \}$ with support in $\Omega_{g}$ such that $u_0 \geq 0$ and $\Vert u_0 \Vert_{p_s^*} = 1$. For each $t>0$ we have
\[
I_{ \lambda}(tu_0) = \frac{t^p}{p} \Vert u_0 \Vert_{s,p}^p + \frac{t^q}{q} \Vert u_0 \Vert_{s,q}^q - \frac{\lambda t^r}{r} \displaystyle\int_{\mathbb{R}^N}gu_0^r \dd x - \frac{t^{p_{s}^*}}{p_{s}^*}, \ t>0.
\]
thus we can see that $I_{\lambda}(tu_{0}) \rightarrow -\infty$ as $t \rightarrow \infty$ and that $I_{\lambda}(tu_{0}) \rightarrow 0$ as $t \rightarrow 0^+ $. These facts imply the existence of a $t_{\lambda} >0$ such that 
\[
\displaystyle\max_{t \geq 0} I_{\lambda}(tu_0) = I_{\lambda}(t_{\lambda} u_0)
\]		
Hence
\begin{align*}
0 &= \frac{d}{dt}\left[I_{\lambda}(tu_0) \right]_{t = t_{\lambda}}\\
&= t_{\lambda}^{p-1}\Vert u_0 \Vert_{s,p}^p + t_{\lambda}^{q-1}\Vert u_0 \Vert_{ s,q}^q - \lambda t_{\lambda}^{r-1}\displaystyle\int_{\mathbb{R}^N}gu_0^r \dd x - t_{\lambda}^{p_{s}^* -1}
\end{align*}
we get
\begin{equation*}
0 < \lambda \displaystyle\int_{\mathbb{R}^N}gu_0^r \dd x = \frac{\vert \vert u_0 \vert \vert_{s,p}^p}{t_{\lambda}^{r-p}} + \frac{\vert \vert u_0 \vert \vert_{s,q}^q}{t_{\lambda}^{r-q}} - t_{\lambda}^{p_{s}^* - r}, \ \mbox{for all} \ \lambda>0.
\end{equation*}
So, $t_{ \lambda} \rightarrow 0$ as $\lambda \rightarrow \infty$. Since $I_{\lambda}(t_{\lambda}u_0) \rightarrow 0$ as $t_{\lambda} \rightarrow 0^+$, there exists $\lambda^* > 0$ such that
\begin{equation*}
\displaystyle\max_{t \geq 0}I_{\lambda}(tu_0) = I_{\lambda}(t_{\lambda}u_0) < \frac{s}{N}S^{\frac{N}{sp}}. \ \mbox{for all} \ \lambda > \lambda^*.
\end{equation*}
The conclusion follows.
\end{proof}
Now assume that
\begin{align}
N > p^2s \ \mbox{and} \ 1< q <\frac{N(p-1)}{N-s} < p \leq \max\{p,p_s^* - \frac{q}{p-1} \} < r < p_s^*. \label{Cond1}
\end{align}	
Let $U$ be a radially symmetric and decreasing minimizer for the Sobolev constant $S=S_{s,p}$. It is know from \cite{Brasco} that there exist constants $c_1,c_2 >0$, and $\theta >1$ such that
\begin{align}
&\frac{c_1}{\vert x \vert^{\frac{N-sp}{p-1}}} \leq U(|x|) \leq \frac{c_1}{\vert x \vert^{\frac{N-sp}{p-1}}}, \ \forall |x| \geq 1, \ \label{Est1} \\
&\frac{U(\theta r)}{U(r)} \leq \frac{1}{2}, \ \forall r \geq 1. \label{Est2}
\end{align}	
By multiplying the function $U$ by an appropriate constant, we can assume that $U$ satisfies the following:
\begin{align*}
&(i) \ (-\Delta_p)^s U = U^{p_s^* -1} \ \mbox{in} \ \mathbb{R}^{N} \\
&(ii) \ \vert \vert U \vert \vert_{s,p}^p = \left\Vert U \right\Vert_{p_s^*}^{p_s^*} = S^{\frac{N}{sp}}.
\end{align*}	
For any $\delta >0$, the function
\begin{equation}
U_{\delta}(x) = \delta^{-\frac{N-sp}{p}}U(|x|/\delta)
\end{equation}
is also a minimizer for $S$, satisfying $(i)$ and $(ii)$. We may assume that $0 \in \Omega_{g}$. For $\delta,R>0$ consider the radially symmetric non-increasing function $\overline{u}_{\delta,R}:[0,\infty) \rightarrow \mathbb{R}$ by
\begin{align*}
\overline{u}_{\delta,R} =\left\{\begin{array}{rc}
U_{\delta}(r),&\mbox{se} \ r\leq R,\\
0, &\mbox{se} \ r\geq \theta R.
\end{array}\right.
\end{align*}
Therefore, we have the following estimates from \cite{Mos}.
\begin{lemma}
For any $R>0$, exist $C= C(N,p,s)>0$ such that for any $\delta \leq \frac{R}{2}$
\begin{align}
&\vert\vert \overline{u}_{\delta,R} \vert\vert_{s,p}^p \leq S^{\frac{N}{sp}} + C(\frac{\delta}{R})^{\frac{n_sp}{p-1}}, \ \label{A1} \\
&\Vert \overline{u}_{\delta,R} \Vert_{p}^p \geq \left\{\begin{array}{rc}
\frac{1}{C} \delta^{sp} log(R/\delta),&\mbox{se} \ N=sp^2,\\
\frac{1}{C}\delta^{sp}, &\mbox{se} \ N > sp^2. 
\end{array}\right.\\
&\Vert \overline{u}_{\delta,R} \Vert_{p_s^*}^{p_s^*} \geq S^{\frac{N}{sp}} - C(\frac{\delta}{R})^{N/(p-1)}. \label{A2}
\end{align}
\end{lemma}
	
Let $\varepsilon > 0$. Take $R>0$ fixed such that $B_{\theta R}(0) \subset \Omega_{g}$ and let us define the function $u_{\varepsilon,R}: [0,\infty) \rightarrow \mathbb{R}$ by
\[
u_{\varepsilon,R}(r) = \varepsilon^{-\frac{N-sp}{p^2}}\overline{u}_{\delta,R}(r), \ \mbox{with} \ \delta = \varepsilon^{\frac{p-1}{p}}.
\]
Therefore applying \eqref{A1}-\eqref{A2} yields
\begin{align}
\Vert u_{\varepsilon,R} \Vert_{s,p}^{p} \leq S^{\frac{N}{sp}} \varepsilon^{-\frac{N-sp}{p}} + O(1).\ \label{A3}
\end{align}
The demonstrations of the following lemma can be found in \cite{IND}.
\begin{lemma}\label{LEIND}
Let $u_{\varepsilon,R}$ be defined as above. Then the following estimates hold for $t\geq 1$,
\begin{align}
&\Vert u_{\varepsilon,R} \Vert_{p_s^*}^p = S^{\frac{N-sp}{sp}}\varepsilon^{-\frac{N-sp}{p}} + O(1). \ \label{A4}\\
&\Vert u_{\varepsilon,R} \Vert_{t}^t \geq 
\left\{\begin{array}{rc}\label{A5}
k\varepsilon^{\frac{N(p-1) - t(N-sp)}{p}} + O(1),&\mbox{se} \ t>\frac{N(p-1)}{N-sp},\\
k|ln\varepsilon| + O(1), &\mbox{se} \ t=\frac{N(p-1)}{N-sp}.\\
O(1), \ &\mbox{se} \ t<\frac{N(p-1)}{N-sp} 
\end{array}\right. \nonumber  \\
\end{align}
and
\begin{align}
\Vert u_{\varepsilon,R}  \Vert_{s,t}^t \leq O(1), \ \mbox{for} \ 1\leq t < \frac{N(p-1)}{N-s} \ \label{A6}
\end{align}
where $k$ is a positive constant independent of $\varepsilon$.
\end{lemma}
	
Now let us show that $\overline{c}_{\lambda} < \frac{s}{N}S^{\frac{n}{sp}}$ ,  for all $\lambda >0$.
\begin{lemma}\label{lema2}
Assume $(g_1)$ and $(g_2)$ and \eqref{Cond1} holds. Then, for any $\lambda >0$ the level $\overline{c}_{\lambda} \in (0,\frac{s}{N} S^{\frac{n}{sp}})$, where $\overline{c}_{\lambda}$ was defined in \eqref{CL}.
\end{lemma}
\begin{proof}
The proof is very similar to that presented in [Lemma 5.4 in \cite{IND}], and hence we will omit it.
\qed
\end{proof}
	
\noindent{\bf Proof of Theorem \ref{EXISTENCE}} We know that the functional $I_{\lambda}$ has the structure of the mountain pass theorem, and from Lemma \ref{LIM} its (PS) sequence is bounded. Let $(u_n) \subset \mathcal{W}$ be a (PS) sequence satisfying 
\begin{equation*}
I_{\lambda}(u_n) \rightarrow c_{\lambda} \,\,\text{and} \,\, I_{\lambda}'(u_n) \rightarrow 0,
\end{equation*}
where $c_\lambda$ is the minimax level of the mountain pass theorem associated with $I_{\lambda}$. Adapting the arguments \cite{Rab,Willen} we concludes that $c_{\lambda} \leq \overline{c}_{\lambda}$. Since that $(u_n)$ is bounded in $\mathcal{W}$, then up to a subsequence one has $u_n \rightharpoonup u$ in $\mathcal{W}$. By Lemma \ref{CQTP} we have $u_n \rightarrow u$ a.e. in $\mathbb{R}^N$.
	
To prove case ($i$) in Theorem \ref{EXISTENCE}, we will use Lemma \ref{C} to get $\lambda^* > 0$ such that $0 < c_{\lambda} \leq \overline{c}_{\lambda}< \frac{s}{N}S^{\frac{N}{sp}}$ for all $\lambda > \lambda^*$. For case ($ii$) we use the Lemma \ref{lema2} to get $0 < c_{\lambda} \leq \overline{c}_{\lambda}< \frac{s}{N}S^{\frac{N}{sp}}$ for all $\lambda >0$.
	
\textbf{Claim:}\label{Afir} Let $u_{n}^{-} = \max \{ -u_n,0 \}$. Then $u_n^{-} \rightarrow 0$ in $\mathcal{W}$, in particular $u_n^{+} \rightarrow u$ a.e. in $\mathbb{R}^{N}$. 
	
Indeed, since $I_{\lambda}'(u_n)u_n^- \rightarrow 0$ then 
\begin{align*}
&\displaystyle\int_{\mathbb{R}^N} \displaystyle\int_{\mathbb{R}^{N}} \left( \frac{J_{p}(u_n(x) - u_n(y))}{\vert x-y\vert^{N+ sp}} + \frac{J_{q}(u_n(x) - u_n(y))}{\vert x-y\vert^{N+sq}}\right)(u_n^{-}(x) - u_{n}^{-}(y)) \dd x\dd y \\
&= \displaystyle\int_{\mathbb{R}^N} ( u_{n}^+ )^{p_{s}^{*}-2}u_{n}^+ u_{n}^{-} \dd x + \lambda \displaystyle\int_{\mathbb{R}^N} g(u_{n}^+)^{r-2}u_{n}^+ u_{n}^{-} \dd x + o(1) = 0
\end{align*}	
Using the elementary inequality, for $m=\{p,q\}$
\begin{align*}
\vert v^{-}(x) - v^{-}(y) \vert^{m} \leq J_{m}(v(x) - v(y))(v^{-}(x) - v^{-}(y)), \ \mbox{for all} \ x,y \in \mathbb{R}^{N}.
\end{align*}
Follows that $u_n^{-} \rightarrow 0$ in $\mathcal{W}$. The claim follows. 
	
Applying the Lemma \ref{Kavian} for $(u_{n}^{+})$ which is bounded in $L^{p_{s}^*}(\mathbb{R}^N)$ results
\begin{align}
&(u_n^+)^{p_{s}^* -1} \rightharpoonup u^{p_{s}^* -1} \ \mbox{in} \ L^{\frac{p_{s}^*}{p_s^* -1}}(\mathbb{R}^N).\label{f1}\\
&(u_n^+)^{r -1} \rightharpoonup u^{r-1} \ \mbox{in} \ L^{\frac{p_{s}^*}{r-1}}(\mathbb{R}^N) \label{f2}
\end{align}	
Let $m \in \{p,q\}$ and denote by
\begin{equation*}
\mathcal{U}_{n}(x,y) = \frac{ \vert u_n(x) - u_n(y) \vert^{m-2}(u_n(x) - u_n(y))}{ \vert x-y \vert^{(N+sm)/m'}} 
\end{equation*}
Since $u_n \rightarrow u$ a.e. in $\mathbb{R}^{N}$ we have
\begin{equation*}
\mathcal{U}_{n}(x,y) \longrightarrow \mathcal{U}(x,y):=  \frac{ \vert u(x) - u(y) \vert^{m-2}(u(x) - u(y))}{ \vert x-y \vert^{( N+ sm)/m'}} \ \mbox{a.e. in } \ \mathbb{R}^{N}.
\end{equation*}
Moreover,
\begin{equation*}
\displaystyle\int_{\mathbb{R}^N}\displaystyle\int_{\mathbb{R}^N} \vert \mathcal{U}_n(x,y) \vert^{m'}\dd x\dd y  \leq \Vert u_n \Vert_{s,m}^m.
\end{equation*}
Then $\left(\mathcal{U}_n\right)$ is bounded in $L^{m'}(\mathbb{R}^{2N})$ for $m \in \{p,q \}$. By Lemma \ref{Kavian} yields,
\begin{align} 
\mathcal{U}_n \rightharpoonup \mathcal{U} \ \mbox{in} \ L^{m'}(\mathbb{R}^{2N}).\label{f3}
\end{align}
Now, for all $\varphi \in \mathcal{W}$, from \eqref{f3} result that
\begin{align}
\displaystyle\int_{ \mathbb{R}^N} \displaystyle\int_{ \mathbb{R}^{N}} \frac{\mathcal{U}_{n}(x,y)(\varphi(x) - \varphi(y))}{ \vert x-y\vert^{(N + sm)/m}} \dd x\dd y \longrightarrow \displaystyle\int_{ \mathbb{R}^N} \displaystyle\int_{\mathbb{R}^{N}} \frac{\mathcal{U}(x,y)(\varphi(x) - \varphi(y))}{\vert x-y\vert^{(N+ sm)/m}} \dd x\dd y. \label{f4}
\end{align}
From \eqref{f1}, \eqref{f2}, and \eqref{f4} results that $I_{\lambda}'(u_n)\varphi \rightarrow I_{\lambda}'(u)\varphi$, for all $\varphi \in \mathcal{W}$ and so $u$ is solution 
(weak) of \eqref{EQ1}. We know that $u\geq 0$. It remains to verify that $u \neq 0$. Let,
\[
\displaystyle\lim_{n\rightarrow \infty}\Vert u_n \Vert_{s,p}^p =:a \geq 0 \ \mbox{and} \ \displaystyle\lim_{n\rightarrow \infty}\Vert u_n \Vert_{s,q}^q =:b\geq 0
\]
and suppose that $u\equiv 0$. Since $I_{\lambda}'(u_n)u_n \rightarrow 0$, we also have
\begin{equation*}
\Vert u_n \Vert_{s,p}^p + \Vert u_n \Vert_{s,q}^q = \lambda \displaystyle\int_{\mathbb{R}^{N}} g (u_n^+)^r \dd x + \displaystyle\int_{\mathbb{R}^{N}} (u_n^+)^{p_{s}^*} \dd x + o(1)
\end{equation*}	
Using the condition $(g_1)$ and the convergence weak $(u_n^+)^{r} \rightharpoonup u^r$ in $L^{\frac{p_{ s}^*}{p_{s}^* -r}}(\mathbb{R}^N)$ we get 
\[
\lambda \displaystyle\int_{\mathbb{R}^{N}} g (u_n^+)^r \dd x \rightarrow 0 .
\]
Thus,
\begin{equation*}
\Vert u_n \Vert_{s,p}^p = a + o(1), \ \  \Vert u_n \Vert_{s,q}^q = b + o(1),  \,\, \text{and} \,\, \Vert u_n \Vert_{p_{s}^*}^{p_{s}^*} = a + b + o(1)
\end{equation*}	
By taking into account that $I_{\lambda}(u_n) \rightarrow c_{\lambda}$, we have 
\begin{align}
\frac{a}{p} + \frac{b}{q} - \frac{a + b}{p_{s}^*} = c_{\lambda} >0 \label{K}
\end{align}
Hence 
\begin{align}
c_{\lambda} &= a \left( \frac{1}{p} - \frac{N- sp}{Np} \right) + b\left( \frac{1}{q} - \frac{1}{p_{s}^*}\right) \\
& \geq a\frac{s}{N}. \label{F}
\end{align}
The equality \eqref{K} shows that $a + b \neq 0$, by definition of $S$ follow that 
\begin{align*}
S(a + b)^{ \frac{p}{p_{s}^*}} \leq a \Rightarrow a>0.
\end{align*}
Thus
\begin{align*}
Sa^{\frac{p}{p_{s}^*}} \leq S(a + b)^{ \frac{p}{p_{s}^*}} \leq a \Rightarrow a \geq S^{\frac{N}{s p}}
\end{align*}	
Then by \eqref{F} we have
\begin{align*}
c_{\lambda}\frac{N}{s} \geq a  \geq S^{\frac{N}{sp}} 
\end{align*}
which is a contradiction, because $c_{ \lambda} < \frac{s}{N}S^{\frac{N}{sp}}.$ This concludes our result. 
\qed

\section{An application}\label{SEC2}
In this section we will apply the regularity results proved in Section \ref{SEC} to show that, if $u \in \mathcal{W}$ satisfies \eqref{EQ1}, then $u \in L^{\infty}(\mathbb{R}^{N}) \cap C^{\alpha}_{loc}(\mathbb{R}^{N})$.

\noindent{\bf Proof of Theorem \ref{APLIC}} Due to the Theorem \ref{BOUNDED}, it is enough to show that $u \in L^{\theta(p_{s}^*-1)}(\mathbb{R}^N)$ for some $\theta > \frac{N}{sp}$. In fact, if this is true, then since $ u \in L^{p_{s}^*}(\mathbb{R}^N)$ and  $g \in L^{t}(\mathbb{R}^N) \cap L^{\infty}(\mathbb{R}^N)$, where $t>0$ is give in $(g_1)$. Using the H\"{o}lder's inequality for $\gamma = \frac{p_{s}^*-1}{r-1} > 1$ and $\gamma'= \frac{p_{s}^*-1}{p_{s}^*-r}$, we obtain
\begin{align*}
\displaystyle\int_{\mathbb{R}^N} |g|^{\theta}u^{\theta(r-1)} \dd x \leq  \Vert g \Vert_{L^{\infty}(\mathbb{R}^N)}^{\frac{\theta(p_{s}^*-1)-p_{s}^*}{p_{s}^*-1}} \Vert g \Vert^{\frac{p^{*}_{s}}{p^{*}_{s}-r}}_{L^{t}(\mathbb{R}^{N})} \left(\displaystyle\int_{\mathbb{R}^N} u^{\theta (p_{s}^*-1)} \dd x \right)^{\frac{r-1}{p_{s}^*-1}} < \infty,
\end{align*}
for $\theta > \frac{N}{sp} > \frac{Np}{Np - N+ sp}$, which implies that $\frac{\theta(p_{s}^*-1)-p_{s}^*}{p_{s}^*-1} >0$. Therefore $f = \lambda gu^{r-1} + u^{p_{s}^* -1} \in L^{\theta}(\mathbb{R}^N)$ with $\theta > \frac{N}{sp}$, which jointly Theorem \ref{BOUNDED} give us $u \in L^{\infty}(\mathbb{R}^N)$.
	
Let us show that $u \in L^{\theta(p_{s}^*-1)}(\mathbb{R}^N)$ for some $\theta > \frac{N}{sp}$. Let $M>0$ and $\beta>1$, and denote as before $u_M= \min\{u,M \}$. Define $h_{\beta,M}(t)= t (\min\{t,M\})^{\beta-1}$. So
\begin{align*}
h_{\beta,M}(t)=\left\{\begin{array}{rc}
t^{\beta},&\mbox{se} \ t\leq M,\\
tM^{\beta-1}, &\mbox{se} \ t\geq M.
\end{array}\right.
\end{align*}
We have that $h_{\beta,M}$ is increasing, continues and has bounded derivative. Hence if $u \in \mathcal{W}$, then $h_{\beta,M}(u) \in \mathcal{W}$. Using the test function $\varphi= h_{\beta,M}(u)$ in equation \eqref{EQ1} we get
\begin{align}
\displaystyle\int_{\mathbb{R}^N}\displaystyle\int_{\mathbb{R}^N} \!\!&\!\left(\frac{J_{p}(u(x) - u(y))}{\vert x-y\vert^{N+ sp }} + \frac{J_{q}(u(x) - u(y))}{\vert x-y\vert^{N+ sq }} \right)\left( h_{\beta,M}(u(x)) - h_{\beta,M}(u(y)) \right) \dd x\dd y \nonumber\\
&= \lambda\displaystyle\int_{\mathbb{R}^N} gu^{r-1}h_{\beta,M}(u) \dd x + \displaystyle\int_{\mathbb{R}^N}u^{p_{s}^*-1}h_{\beta,M}(u)\dd x \label{De1}\\
&= \lambda \displaystyle\int_{\mathbb{R}^N} gu^{r}u_M^{\beta -1} \dd x +   \displaystyle\int_{\mathbb{R}^N} u^{p_{s}^*} u_M^{\beta -1} \dd x =: J_1 + J_2 \nonumber
\end{align}
where 
\begin{align*}
J_{1} &:= \lambda \displaystyle\int_{\mathbb{R}^N} gu^{r}u_M^{\beta -1} \dd x \\
J_{2} &:= \displaystyle\int_{\mathbb{R}^N} u^{p_{s}^*} u_M^{\beta -1} \dd x
\end{align*}
The term $J_2$ was estimated in \cite{BP} by
\begin{align}
J_2 \leq K_0^{\beta -1}\displaystyle\int_{\mathbb{R}^N} u^{p_{s}^*} \dd x + \left(\displaystyle\int_{\{u\geq K_0 \}} u^{p_{s}^*} \dd x\right)^{\frac{p_{s}^*-p}{p_{s}^*}} \left(\displaystyle\int_{\mathbb{R}^N} u^{p_{s}^*} u_M^{(\beta -1)\frac{p_{s}^*}{p}}\dd x\right)^{\frac{p}{p_{s}^*}},
\end{align}	
where $K_0 >1$ is a given constant. To estimate $J_1$, since $u_M \leq u$ we get
\begin{align*}
\lambda \displaystyle\int_{\{u<K_0\}}gu^ru_M^{\beta -1}\dd x \leq \lambda K_0^{\beta-1} \Vert g\Vert_{L^{t}(\mathbb{R}^N)}\left(\displaystyle\int_{\mathbb{R}^N}u^{p_{s}^*}\dd x \right)^{\frac{r}{p_{s}^*}}.
\end{align*}
On the other hand since $K_0 >1$ and $r< p_{s}^*$, using that $g \in L^{\infty}(\mathbb{R}^N)$ and Holder inequality
\begin{align*}
\lambda \displaystyle\int_{\{u\geq K_0\}}\!\!gu^ru_M^{\beta -1}\dd x &\leq \lambda \displaystyle\int_{\{u\geq K_0\}}\!\!gu^{p_{s}^*}u_M^{\beta -1}\dd x \leq \lambda \Vert g \Vert_{L^\infty(\mathbb{R}^N)} \displaystyle\int_{\{u \geq K_0\}}\!\!u^{p_{s}^*}u_M^{\beta -1}\dd x \\
&\leq C \left(\displaystyle\int_{\{u\geq K_0\}} u^{p_{s}^*}\dd x \right)^{\frac{p_{s}^*-p}{p_{s}^*}} \left(\displaystyle\int_{\mathbb{R}^N}u^{p_{s}^*}u_M^{(\beta-1)\frac{p_{s}^*}{p}} \dd x \right)^{\frac{p}{p_{s}^*}}.
\end{align*}
Then
\begin{align*}
J_1 \leq C K_0^{\beta-1} \vert \vert u \vert \vert_{L^{p_{s}^*}(\mathbb{R}^N)}^r + C \left(\displaystyle\int_{\{u\geq K_0\}} u^{p_{s}^*}\dd x \right)^{\frac{p_{s}^*-p}{p_{s}^*}}\!\! \left(\displaystyle\int_{\mathbb{R}^N}u^{p_{s}^*}u_M^{(\beta-1)\frac{p_{s}^*}{p}} \dd x \right)^{\frac{p}{p_{s}^*}}.
\end{align*}
Let 
\begin{align}
G_{\beta,M}(t) = \displaystyle\int_{0}^{t} (h_{\beta,M}'(\tau))^{\frac{1}{p}} \dd\tau \geq \frac{p}{\beta + p -1}t(\min\{t,M\})^\frac{\beta -1}{p}\label{Est}.
\end{align}
By Sobolev inequality \eqref{S}, and Lemma \ref{Le} we can see that
\begin{align*}
S \left(\displaystyle\int_{\mathbb{R}^N}\vert G_{s,M}(u(x)) \vert^{p_{s}^*} \dd x \right)^{\frac{p}{p_{s}^*}} &\leq \displaystyle\int_{\mathbb{R}^N} \displaystyle\int_{\mathbb{R}^N} \frac{\vert G_{\beta,M}(u(x)) - G_{\beta,M}(u(y)) \vert^p }{\vert x-y\vert^{N+ sp}} \dd x\dd y \\
&\leq J_1 + J_2 
\end{align*}
Consequently we have
\begin{align*}
S \left(\displaystyle\int_{\mathbb{R}^N}\vert G_{s,M}(u(x)) \vert^{p_{s}^*} \dd x \right)^{\frac{p}{p_{s}^*}}&\leq C_1K_0^{\beta -1} \left( \vert \vert u \vert \vert_{L^{p_{s}^*}(\mathbb{R}^N)}^r + \vert \vert u \vert \vert_{L^{p_{s}^*}(\mathbb{R}^N)}^{p_{s}^*}\right) \\
+ C_2 &\left(\displaystyle\int_{\{u\geq K_0\}} u^{p_{s}^*}\dd x \right)^{\frac{p_{s}^*-p}{p_{s}^*}}\!\! \left(\displaystyle\int_{\mathbb{R}^N}u^{p_{s}^*}u_M^{(\beta-1)\frac{p_{s}^*}{p}} \dd x \right)^{\frac{p}{p_{s}^*}}.
\end{align*}
From \eqref{Est}, and the above inequality, we get
\begin{align}\label{De3}
S\left(\frac{p}{p+\beta -1}\right)^p\!\! &\left(\displaystyle\int_{\mathbb{R}^N}\!\! u^{p_{s}^*} u_M^{(\beta-1)\frac{p_{s}^*}{p}} \dd x \right)^{\frac{p}{p_{s}^*}}\!\! \leq C_1K_0^{\beta -1}\!\! \left( \vert \vert u \vert \vert_{L^{p_{s}^*}(\mathbb{R}^N)}^r\!\! + \vert \vert u \vert \vert_{L^{p_{s}^*}(\mathbb{R}^N)}^{p_{s}^*}\right) \\
&+ C_2 \left(\displaystyle\int_{\{u\geq K_0\}} u^{p_{s}^*}\dd x \right)^{\frac{p_{s}^*-p}{p_{s}^*}} \left(\displaystyle\int_{\mathbb{R}^N}u^{p_{s}^*}u_M^{(\beta-1)\frac{p_{s}^*}{p}} \dd x \right)^{\frac{p}{p_{s}^*}}.\nonumber
\end{align}	
Fixing $\theta > \frac{N}{s p}$, and choosing $\beta > 1$ such that
\[
(\beta - 1)\frac{p_{s}^*}{p} + p_{s}^* = \theta(p_{s}^*-1) \ \mbox{i.e.} \  \beta = p\theta \frac{(p_{s}^*-1)}{p_{s}^*} - (p-1).
\]
Now, choose $K_0 = K_0(\beta,u) > 0$ such that
\[
\left(\displaystyle\int_{\{ u \geq K_0 \}}u^{p_{s}^*} \dd x \right)^{\frac{p_{s}^*-p}{p_{s}^*}} \leq \frac{S}{2}\left(\frac{p}{\beta +p -1}\right)^{p}
\]
Hence from \eqref{De3} we get
\[
\left(\displaystyle\int_{\mathbb{R}^N}u_M^{\theta(p_{s}^*-1)} \dd x \right)^{\frac{p}{p_{s}^*}} \leq C\left(\frac{p + \beta -1}{p} \right)^{p}K_0^{\beta -1} \left(\vert \vert u \vert \vert_{L^{p_{s}^*}(\mathbb{R}^N)}^r + \vert \vert u \vert \vert_{L^{p_{s}^*}(\mathbb{R}^N)}^{p_{s}^*}\right).
\]
As a result $u \in L^{\theta(p_{s}^*-1)}(\mathbb{R}^N)$. Thus $u \in L^{\infty}(\mathbb{R}^{N})$ and since $g \in L^{\infty}(\mathbb{R}^{N})$ follow that $f = \vert u \vert^{p^{*}_{s} - 2}u + \lambda g\vert u \vert^{r-2}u \in L^{\infty}(\mathbb{R}^{N})$. Therefore, by Theorem \ref{BOUNDED} results $u \in C^{\alpha}_{loc}(\mathbb{R}^{N})$.
\qed

\end{document}